\documentclass[12pt, a4paper]{amsart}
\usepackage{enumerate, longtable}
\usepackage{amsmath, amscd, amsfonts, amsthm, amssymb, latexsym, comment, stmaryrd, graphicx}
\usepackage{xcolor}
\usepackage[all]{xy}
\usepackage{fullpage}
\usepackage[
        colorlinks,
        backref,
]{hyperref}

\usepackage{amstext} 
\usepackage{array}

\theoremstyle{plain}
\newtheorem{theorem}{Theorem}[section]
\newtheorem{proposition}[theorem]{Proposition}
\newtheorem{lemma}[theorem]{Lemma}
\newtheorem{question}[theorem]{Question}
\newtheorem{conjecture}[theorem]{Conjecture}
\newtheorem{corollary}[theorem]{Corollary}

\theoremstyle{definition}
\newtheorem{definition}[theorem]{Definition}
\newtheorem{remark}[theorem]{Remark}

\newtheorem{example}[theorem]{Example}

\numberwithin{equation}{section}

\newcommand{\Fr}{{\rm Fr}}

\newcommand{\Gal}{{\rm Gal}}

\newcommand{\Mon}{\mathrm{Mon}}

\newcommand{\F}{\mathbb{F}}
\newcommand{\Z}{\mathbb{Z}}
\newcommand{\ZZ}{\mathbb{Z}}
\newcommand{\Q}{\mathbb{Q}}

\newcommand{\C}{\mathbb{C}}
\newcommand{\FF}{\mathbb{F}}

\def\FF{\mathbb{F}}

\def\AA{\mathbb{A}}

\def\AA{\mathbb{A}}

\def\disc{\Delta} 

\def\Disc{\disc}

\newcommand{\Fb}{\overline{\F}}

\newcommand{\A}{\mathbb{A}}
\renewcommand{\P}{\mathbb{P}}
\newcommand{\Spec}{\mathrm{Spec}}
\newcommand{\ab}{\mathrm{ab}}
\newcommand{\lc}{\mathrm{lc}}
\newcommand{\Res}{\mathrm{Res}}

\newcommand{\HH}{H}

\newcommand{\case}[1]{\vspace{0.3cm}\noindent\framebox{#1.}}
\newcommand{\maincase}[1]{{\setlength{\fboxrule}{1.5pt}\vspace{0.3cm}\noindent\framebox{#1.}}}

\newcommand{\ram}{r}
\renewcommand{\gcd}{}

\begin{document}

\title[Minimal ramification problem for function fields]{The minimal ramification problem
for\\ rational function fields over finite fields}

\author{Lior Bary-Soroker}
\address{Raymond and Beverly Sackler School of Mathematical Sciences, Tel Aviv University, Tel Aviv 69978, Israel}
\email{barylior@tauex.tau.ac.il}
\author{Alexei Entin}
\address{Raymond and Beverly Sackler School of Mathematical Sciences, Tel Aviv University, Tel Aviv 69978, Israel}
\email{aentin@tauex.tau.ac.il}
\author{Arno Fehm}
\address{Technische Universit\"at Dresden, Fakult\"at Mathematik, Institut f\"ur Algebra, 01062 Dresden, Germany}
\email{arno.fehm@tu-dresden.de}


\begin{abstract}
We study the minimal number of ramified primes in Galois extensions of rational function fields over finite fields with prescribed finite Galois group.
In particular,
we obtain a general conjecture in analogy with the 
well studied case of number fields,
which we establish
for abelian, symmetric and alternating groups in many cases.
\end{abstract}

\maketitle

\section{Introduction}

\noindent
Let $G$ be a nontrivial finite group.
We consider the minimal number of ramified primes in 
Galois extensions $L/K$ with Galois group $G$,
in various situations:

In the case $K=\mathbb{Q}$,
the primes are as usual equivalence classes of absolute values,
corresponding to the prime numbers and the usual absolute value as the prime at infinity.
Let $\ram_K(G)$ denote the minimal number of ramified primes\footnote{The infinite prime is said to ramify in $L/\mathbb{Q}$
if $L$ is not totally real.}
 in a Galois extension $L/K$ with Galois group $G$.
Based on class field theory, Boston and Markin came to the following conjecture:
\begin{conjecture}[Boston--Markin \cite{BM}]\label{conj:BM}
For every nontrivial finite group $G$,
$$
 \ram_\mathbb{Q}(G) = \max\left\{d(G^{\rm ab}),1\right\},
$$
where $d(G)$ is the smallest cardinality of a set of 
generators of $G$. 
\end{conjecture}

While this 
was verified by Boston and Markin for abelian groups and for groups of order up to 32,
and since then was proven for certain solvable groups \cite{Plans,KisilevskySonn,KisilevskySonn2010,KisilevskyNeftinSonn},
it is widely open for most other groups.
For example, for the symmetric group $G=S_n$, $n>1$, this conjecture predicts realizations
with only one ramified prime.
While this can of course be verified for small $n$,
the best general bound to date is the recent result of Bary-Soroker and Schlank \cite{BSS}, who obtain $\ram_\mathbb{Q}(S_n)\leq 4$ for all $n$.

In the case where $K=k(T)$ is a rational function field over a field $k$, the primes are equivalence classes of ultrametric absolute values on $K$ trivial on $k$,
which correspond to the monic irreducible polynomials $P\in k[T]$ 
and the degree valuation as the prime at infinity.
As before we let $\ram_K(G)$ be
the minimal number of ramified primes
in Galois extensions $L/K$ with Galois group $G$,
where in addition we require that $L/K$ is {\em geometric}, i.e.~$k$ is algebraically closed in $L$.
The geometric situation $k=\mathbb{C}$ is well understood:

\begin{theorem}[Riemann Existence Theorem]\label{thm:RET}
For every nontrivial finite group $G$,
$$
 \ram_{\mathbb{C}(T)}(G) = d(G)+1.
$$
\end{theorem}
In positive characteristic, wild ramification
leads to smaller bounds:

\begin{theorem}[Abhyankar conjecture for the line, \cite{Ray94,Har94}]\label{thm:Abhyankar}
For every prime number $p$ and every nontrivial finite group $G$,
$$
 \ram_{\overline{\mathbb{F}}_p(T)}(G) = d(G/p(G)) + 1,
$$
where $p(G)$ is the subgroup of $G$ generated by the $p$-Sylow subgroups of $G$.
\end{theorem}

For the global field $K=\mathbb{F}_q(T)$,
in analogy with Conjecture \ref{conj:BM},
we thus conjecture the following:

\begin{conjecture}\label{conj}\label{bm}
For every prime power $q=p^\nu$ and every nontrivial finite group $G$,
$$
 \ram_{\mathbb{F}_{q}(T)}(G) = \max\left\{d((G/p(G))^{\rm ab}),1\right\}.
$$
\end{conjecture}

We note that $(G/p(G))^{\rm ab}=G^{\rm ab}/p(G^{\rm ab})$,
and that the close analogy with the case of $K=\mathbb{Q}$ builds in an essential way on our restriction 
that 
the extension $L$ of $K=\mathbb{F}_q(T)$ with Galois group $G$ is geometric.
In Section \ref{sec:abelian} we prove this conjecture for abelian groups,
thereby showing that the right hand side is always a lower bound for $\ram_{\mathbb{F}_{q}(T)}(G)$.
There we also discuss the contribution of De Witt \cite{DeWitt},
who proves Conjecture \ref{conj} for certain solvable groups.

\begin{remark}\label{rem:joachim} 
Another potential constraint on $\ram_{\mathbb{F}_{q}(T)}(G)$ was pointed out to us by Joachim K\"onig. 
For a finite geometric Galois extension $L$ of $K=\mathbb{F}_q(T)$,
the conjugacy classes of inertia groups generate $\Gal(L/K)$,
and each inertia group is cyclic-by-$p$, i.e.~an extension of a cyclic group by a $p$-group (see \cite[Corollary IV.2.4]{Ser79}). Thus Conjecture \ref{conj} is possibly true only if any finite group $G$ is generated by $\max\{d((G/p(G))^{\rm ab}),1\}$ conjugacy classes of cyclic-by-$p$ groups. 
We give the proof of this elementary but nontrivial group-theoretic fact also in Section \ref{sec:generators}. 
A similar remark applies to Conjecture \ref{conj:BM}, where Lemma~\ref{lem:gen_by_cyc} ensures there is no group-theoretic obstruction.
\end{remark}

In the main part of this work (Sections \ref{sec:tworamified}  and \ref{sec:pleqn}) we prove various results for certain non-solvable groups, both conditional and unconditional,
with a focus on symmetric and alternating groups.
Note that Conjecture \ref{conj}
predicts 
$$
 \ram_{\mathbb{F}_q(T)}(S_n)=1,\quad  \ram_{\mathbb{F}_q(T)}(A_n)=1
$$
for every $n>2$ and every prime power $q=p^\nu$.
For example, we obtain the following for $S_n$ and $A_n$:

\begin{theorem}\label{thm:Sn}
Let $n\geq 2$ and $q=p^\nu$ a prime power.
Then 
\begin{enumerate}
\item $\ram_{\mathbb{F}_q(T)}(S_n)\leq 2$, and 
\item $\ram_{\mathbb{F}_q(T)}(S_n)=1$ in each of the following cases:
\begin{enumerate}
\item $p<n-1$ or $p=n$ or $p=2$
\item $q\equiv 1\mbox{ mod } 4$
\item $q>(2n-3)^2$
\item The function field analogue of Schinzel's hypothesis H (Conjecture \ref{conj:SchinzelFF}) holds for $\mathbb{F}_q(T)$.
\end{enumerate}
\end{enumerate}
\end{theorem}

\begin{theorem}\label{thm:An}
Let $n\geq 3$ and $q=p^\nu$ a prime power.
\begin{enumerate}
\item If $p>2$ or $\F_q\supseteq\F_4$, then $\ram_{\mathbb{F}_q(T)}(A_n)\leq 2$, and 
\item $\ram_{\mathbb{F}_q(T)}(A_n)=1$ in each of the following cases:
\begin{enumerate}
\item $2<p<n-1$ or $p=n$ or $p=n-1,\F_q\supseteq\F_{p^2}$ or $p=3,n=4$
\item $p=2$ and $\F_q\supseteq\F_4$ or $10\neq n\ge 8,n\equiv 0,1,2,6,7\pmod 8$ 
\item $n=10$, $n=11$ or $n\ge 13$, and the function field analogue of Schinzel's hypothesis H (Conjecture \ref{conj:SchinzelFF}) holds for $\mathbb{F}_q(T)$.
\end{enumerate}
\end{enumerate}
\end{theorem}

One surprising prediction of Conjecture \ref{conj:BM} is that for an arbitrary product $G=A_{n_1}\times\ldots\times A_{n_m},n_i\ge 5$, there exists a Galois extension $L/\mathbb Q$ with $\Gal(L/\mathbb Q)=G$ which is ramified over a single prime. This is currently wide open even in the case $G=A_5\times\ldots\times A_5$. We will establish Conjecture \ref{conj} for arbitrary products of the groups appearing in the statements of Theorems \ref{thm:Sn}(2a), \ref{thm:An}(2a) and \ref{thm:An}(2b):

\begin{theorem}\label{thm:main1} Let $G=G_1\times\ldots\times G_m$ be a product with each $G_i$ being one of the following groups:
\begin{enumerate}\item $S_n,n\ge p$ with $n\neq p+1$ or $p=2$.
\item $A_n, n\ge p>2$ with $n\neq p+1$ or $\F_q\supseteq\F_{p^2}$ or $p=3$.
\item $A_n$ if $p=2$ with either $\F_q\supseteq\F_4$ or $10\neq n\ge 8,n\equiv 0,1,2,6,7\pmod 8.$
\end{enumerate}
Then $\ram_{\mathbb{F}_q(T)}(G)=\max\{d((G/p(G))^{\rm ab}),1\}$, i.e.\ Conjecture \ref{bm} holds for $G$.
\end{theorem}

En route to proving Theorem~\ref{thm:main1} we will obtain evidence for the following conjecture  of Abhyankar (see \cite[\S 16]{Abh01},\cite[\S 5]{HOPS}, a weaker form was stated in \cite[Conjecture 9.2(C)]{Abh95})
which is an analogue over $\mathbb{F}_q$ of the Abhyankar conjecture for the affine line (which is the case $G=p(G)$ of Theorem~\ref{thm:Abhyankar}):

\begin{conjecture}[Abhyankar's arithmetic conjecture for the affine line]\label{conj:abhyankar} 
Let $G$ be a finite group which is cyclic-by-quasi-$p$, meaning that $G/p(G)$ is cyclic (see Definition~\ref{def:quasip} below).
Then for every power $q$ of $p$ there exists a Galois extension $L/\F_q(T)$ (not necessarily geometric)
ramified only over the infinite prime
such that $\Gal(L/\F_q(T))=G$.
\end{conjecture}

We will show that Conjecture \ref{conj:abhyankar} holds for the alternating group $A_n$ provided $n\ge p>2,n\neq p+1$ 
(note that for $n>3$ the group $A_n/p(A_n)$ is cyclic only if $n\ge p$).

\begin{theorem}\label{thm:abhyankar} Assume $n\ge p>2$. If $n=p+1$ assume additionally that $\F_q\supseteq\F_{p^2}$. Then there exists a Galois extension $L/\F_q(T)$ ramified only over the infinite prime
with $\Gal(L/\F_q(T))=A_n$.
\end{theorem}
Theorem~\ref{thm:abhyankar} will follow from Theorem~\ref{thmlist} below. For some values of $q,n$ the theorem above is implicit in the work of Abhyankar \cite{Abh92,Abh93}, but not in the above generality.

To obtain these results we apply a variety of tools and recent results
from Galois theory, finite group theory (including the Classification of Finite Simple Groups for some of the results; see Remark \ref{remark:cfsg_details} for the list of cases which can be proved without the CFSG), analytic number theory
and number theory over function fields,
and we also employed the help of a computer to eliminate some small exceptional cases.
See the summary in Section \ref{sec:summary} 
for the proof of Theorems \ref{thm:Sn} and \ref{thm:An},
and the corresponding sections  for the full strength and full generality of the results.
Using our results for $\mathbb{F}_q(T)$ we also obtain a conditional result for $\mathbb{Q}$ (Theorem~\ref{thm:S_n_over_Q}):

\begin{theorem}
Schinzel's hypothesis H (Conjecture \ref{conj:Schinzel}) implies that
$\ram_\mathbb{Q}(S_n)=1$ for every $n>1$,
i.e.~Conjecture \ref{conj:BM} holds for all symmetric groups.
\end{theorem}

This improves upon a result of Plans (see Remark \ref{remark:plans} below), which shows that under Schinzel's hypothesis H there exists a Galois extension $L/\mathbb Q$ with Galois group $S_n$ and ramified over a single finite prime and the infinite prime, whereas our construction has no ramification at the infinite prime.

\section{Preliminaries and notation}
\noindent
We start by collecting a few definitions and statements that we will use throughout the paper.

\subsection{Function fields}

For a thorough introduction to the subject see \cite{Ros02,Stichtenoth}. Here we only recall some terminology, notation and a few basic facts. 
Let $k$ be a base field. 
A function field (of one variable) over $k$ is a finite extension of the rational function field $k(T)$. 
A \emph{global function field} is a function field over a finite base field $k=\F_q$ (this will be our main case of interest).

A function field $K/k$ is {\em regular}
if $K$ is linearly disjoint over $k$ from the algebraic closure $\bar{k}$ of $k$,
which in case that $k$ is perfect simply means that
$k$ is relatively algebraically closed in $K$.
If $K/k$ is a regular function field, there exists a unique geometrically integral non-singular projective curve $C/k$ such that $K$ is the function field of $C$. Conversely the function field of a curve as above is a regular function field over $k$.

Let $K/k$ be a regular function field and $C$ the corresponding curve. 
A \emph{prime} $P$ of $K/k$ is an
equivalence class of nontrivial absolute values on $K$ trivial on $k$. Each such absolute value corresponds to a discrete valuation $v$,
and it corresponds to a prime divisor on $C$. 
If $P$ is a prime of $K$, the residue field $k(P)$ at $P$ is a finite extension of $k$ and $\deg P=[k(P):k]$ is called the \emph{degree} of $P$. By a \emph{point} or \emph{geometric point} of $C$ we will always mean a closed point of $C\times_k\bar k$. 
For an extension $L/k$ we denote by $C(L)$ the set of $L$-rational points on $C$.

\renewcommand{\div}{\mathrm{div}}

The rational function field $k(T)$ corresponds to the projective line $\P^1_k$ and has primes of two types: the \emph{finite primes} of the form $\div(f)_0$ (the zero divisor of a rational function $f$) where $f\in k[X]$ is an irreducible polynomial (such a prime has degree $\deg f$) and the \emph{infinite prime} $\infty=\div(1/T)_0$ which has degree 1.

Let $L/K$ be a finite extension of function fields over $k$. It is called \emph{regular}, or \emph{geometric}, if $L$ is also a regular function field over $k$. Regular extensions correspond to finite covers of the associated curves. 
If $P$ is a prime of $K$ and $Q$ a prime of $L$ lying over $P$, we denote by $e(Q/P)$ the ramification index of $Q$ over $P$. 
We say that $L/K$ is ramified (resp. wildly ramified) at $Q$ if $e(Q/P)>1$ (resp.\ $\mathrm{char}(k)|e(Q/P)$). We say that $L/K$ is ramified (resp. wildly ramified) \emph{over} $P$ if there exists a prime $Q$ of $L$ lying over $P$ at which the extension is ramified (resp. wildly ramified).
Ramification which is not wild is said to be \emph{tame}.
If $L/K$ is tamely ramified over at least one prime, then $L/K$ is separable.

If $C_L$ and $C_K$ are the underlying curves then the extension $L/K$ corresponds to a finite morphism $w\colon C_L\to C_K$. 
This correspondence defines an equivalence of categories between regular function fields over $k$ and absolutely irreducible non-singular projective curves defined over $k$. 
If $P$ is a prime divisor of $C_L$, and $Q=w(P)$ (a prime divisor of $C_K$), we say that $w$ is ramified (resp. tamely ramified) at $P$ whenever the extension $L/K$ is ramified (resp. tamely ramified) at the corresponding prime of $L$ which we identify with $P$. 
In this case we also call the geometric points corresponding to $P$ {\em ramification points} of $w$,
and the geometric points corresponding to $Q$ {\em branch points} of $w$.

\subsection{Ramification theory}

Let $K$ be a field. For a separable polynomial $f\in K[X]$ of degree $n$ we define its Galois group $\Gal(f/K)$ to be the Galois group of its splitting field over $K$.
We will always interpret the Galois group ${\rm Gal}(f/K)$ as a subgroup of the symmetric group $S_n$ via its action on the roots of $f$. The embedding into $S_n$ is well-defined up to conjugation.

\begin{lemma}\label{lemcycle} 
Let $K$ be a function field over an algebraically closed field $k$
of characteristic $p\geq 0$,
let $f\in K[X]$ be a separable irreducible polynomial and let $L=K(\alpha)$ where $\alpha$ is a root of $f$. 
Let $P$ be a prime (finite or infinite) of $K$ and let $Q_1,\ldots,Q_r$ be the primes of $L$ lying over $P$. 
Let $e_i=e(Q_i/P)$ be the ramification indices.
\begin{enumerate}\item[(i)]
Assume that $p\nmid e_i$ for all $i$. 
Then $\Gal(f/K)$ contains a permutation which is a product of $r$ disjoint cycles of lengths $e_1,\ldots,e_r$.
\item[(ii)]
Assume that
$p\nmid e_i$ and $(e_i,e_1)=1$ for $2\le i\le r$, 
and either $e_1=p$ or $p\nmid e_1$. 
Then $\Gal(f/K)$ contains a cycle of length $e_1$.\end{enumerate}\end{lemma}

\begin{proof} See \cite[\S 3, Generalized Cycle Lemma]{Abh93}.\end{proof}

\begin{lemma}[Abhyankar's Lemma] \label{lem:abhyankar} 
Let $k$ be perfect, $K/k$ a function field, and
$L,M$ finite separable extensions of $K$. 
Let $P$ be a prime of $K$, $Q$ a prime of $L$ lying over $P$, $P'$ a prime of $M$ lying over $P$. 
Assume that 
\begin{enumerate}
\item $Q$ is tamely ramified over $P$, and
\item $e(Q/P)$ divides $e(P'/P)$.
\end{enumerate}
Then every prime $Q'$ of $LM$ lying over both $Q$ and $P'$ is unramified over $M$, i.e.~$e(Q'/P')=1$.
\end{lemma}

\begin{proof} 
This follows from \cite[Theorem 3.9.1]{Stichtenoth}
by the multiplicativity of the ramification index.
\end{proof}

\subsection{Group theory}
We denote by $C_n$, $D_n$, $S_n$, $A_n$ 
the cyclic, dihedral, symmetric respectively alternating group of degree $n$.
We will need the following group-theoretic result of Jones, which relies on the Classification of Finite Simple Groups (CFSG) for its proof.

\begin{theorem}[Jones]\label{thm:jones} Let $G\leqslant S_n$ be a primitive permutation group containing a cycle of length $\ell>1$.
\begin{enumerate}\item[(i)]If $\ell\le n-3$ then $G\geqslant A_n$.
\item[(ii)]If $\ell=n-2$ is prime then either $G\geqslant A_n$ or $\ell=2^k-1$ is a Mersenne prime and ${PGL}_2(2^k)\leqslant G\leqslant{P\Gamma L}_2(2^k)$ with its standard action on $\P^1(\F_{2^k})$.
\end{enumerate}
\end{theorem} 

\begin{proof} This is a direct consequence of \cite[Corollary 1.3]{Jon14}.\end{proof}

\begin{remark}\label{remark:jordan}
In the case that $\ell$ is prime Theorem~\ref{thm:jones}(i) is a classical (and elementary) result of Jordan \cite[Theorem 3.3E]{DiMo96}. Most (though not all) of the applications in Section~\ref{sec:pleqn} only require this weaker version. In Section~\ref{sec:tworamified} Jones's theorem is used in its full form.
\end{remark}

We will also make use of the following elementary fact:

\begin{lemma}\label{lem:primtrans}
Let $G\leqslant S_n$ be a primitive group containing a transposition. Then $G=S_n$.
\end{lemma}

\begin{proof} See \cite[Theorem 13.3]{Wie64}.\end{proof}

\subsection{Discriminants}

Let $K$ be a field. 
For a polynomial $f\in K[X]$ of degree $n$ we denote its leading coefficient by $\lc(f)$,
and if $\alpha_1,\dots,\alpha_n\in\overline{K}$ are the zeros of $f$ (listed with multiplicity), the discriminant of $f$ is defined by
\begin{equation}\label{eq:discdefinition}
\disc(f) \;=\; \lc(f)^{2n-2}\prod_{i<j}(\alpha_i-\alpha_j)^2 \in K.
\end{equation}
Sometimes we will use the notation $\disc_X(f)$ to indicate the variable with respect to which the discriminant is taken.

If $0\neq f,g\in K[X]$ have roots $\alpha_1,\ldots,\alpha_m$ respectively $\beta_1,\ldots,\beta_n$, their resultant is defined by
\begin{equation}\label{eq:resdefinition}
\Res(f,g)\;=\;\lc(f)^{{\rm deg}(g)}\lc(g)^{{\rm deg}(f)}\prod_{i=1}^m\prod_{j=1}^n(\alpha_i-\beta_j)\;=\;
\lc(f)^{{\rm deg}(g)}\prod_{i=1}^m g(\alpha_i)\in K.
\end{equation}

We will now summarize some formulae which are useful for computing discriminants, valid in any characteristic. See \cite[\S 3]{CCG} for the proofs (but note that there $\Disc(f)$ is defined without the factor $\lc(f)^{2n-2}$). First we have
\begin{eqnarray} 
\label{eq:disc}\Disc(f)&=&(-1)^{\frac{\deg f\cdot(\deg f-1)}2+\deg f\cdot\deg f'}\lc(f)^{\deg f-\deg f'-2}\Res(f',f)\\
&=&(-1)^{\frac{\deg f\cdot(\deg f-1)}2+\deg f\cdot\deg f'}\lc(f)^{\deg f-\deg f'-2}\lc(f')^{\deg f}\prod_{f'(\rho)=0}f(\rho)
\label{eq:disc2}
\end{eqnarray}
provided $f'\neq 0$.
Here $f'$ is the formal derivative of $f$ and the product is over all roots of $f'$ counted with multiplicity.
Now let $f,g,h\in K[X]$.
From (\ref{eq:discdefinition}) and (\ref{eq:resdefinition}) we see that
\begin{equation}\label{eq:discprod}
 \Disc(fg)=\Disc(f)\Disc(g)\Res(f,g)^2
\end{equation} 
The resultant is bi-multiplicative, i.e.\
\begin{equation}\label{eq:resbimult}\Res(f,gh)\;=\;\Res(f,g)\Res(f,h),\quad\Res(fg,h)\;=\;\Res(f,h)\Res(g,h)
\end{equation}
and satisfies a symmetry property 
\begin{equation}\label{eq:ressym}\Res(f,g)=(-1)^{\deg f\cdot\deg g}\Res(g,f).\end{equation}
Finally if $g\equiv h\pmod f$ then
\begin{equation}\label{eq:resalt}\Res(f,g)=\lc(f)^{\deg g-\deg h}\Res(f,h).\end{equation}

\begin{lemma}\label{lem:discgalois}
If $\mathrm{char}(K)\neq 2$ and $f\in K[X]$ is a separable polynomial of degree $n$ then $\Gal(f/K)\subseteq A_n$ if and only if $\Disc(f)$ is a square in $K$. 
In particular, ${\rm Gal}(f/K(\sqrt{\Disc(f)}))={\rm Gal}(f/K)\cap A_n$.
\end{lemma}
\begin{proof} This standard fact is easy to deduce from (\ref{eq:discdefinition}).\end{proof}

\begin{lemma}\label{lem:Dedekind}
Let $K=k(T)$ and $f\in k[T,X]$ irreducible and monic in $X$.
If $P\in k[T]$ is a finite prime which is ramified in the splitting field of $f$ over $K$, then $P|\Disc(f)$. 
Moreover if $P$ divides $\Disc(f)$ exactly once, then in the extension $K(\alpha)/K$, where $\alpha$ is a root of $f$,
there is exactly one ramified prime lying over $P$, and its ramification index is 2. 
\end{lemma}

\begin{proof} 
This standard fact can be easily deduced from the properties of the different and its connection to the discriminant, see e.g. \cite[Proposition 7.9 and Corollary 7.10.2]{Ros02}.
\end{proof}

\subsection{Critical points and critical values}
Let $k$ be a perfect field
and let $w\in k(X)$ be a non-constant rational function. 
One can view $w$ also as a morphism $w\colon\P^1\to\P^1$ defined over $k$.
For convenience we will use $X$ for the variable in the domain of $w$, which we denote $\P^1_X$ and $U$ in the range of $w$, which we denote $\P^1_U$ and we adopt a similar convention for the affine line $\A^1=\P^1\setminus\{\infty\}$. A point $\alpha\in\P^1_X(\bar k)$ is called a \emph{critical point} of $w$ if $w$ is ramified at $\alpha$, 
i.e.~if the corresponding prime of $\bar{k}(X)$ is ramified in the extension of function fields $\bar{k}(X)/\bar{k}(U)$ induced by $w$. 
A point $\beta\in\P^1_U(\bar k)$ is called a \emph{critical value} for $w$ if there exists a critical point $\alpha\in\P^1_X(\bar k)$ such that $\beta=w(\alpha)$, or equivalently, if $w$ is branched over $\beta$. 

The following lemma gives a useful description of the critical points and values of a rational function in terms of derivatives and discriminants.

\begin{lemma}\label{lem:crit}
Let $w=\frac{f}{c}$ with $f,c\in k[X]$, $(f,c)=1$, be a non-constant rational function.
\begin{enumerate}
\item[(i)]
Let $\alpha\in\bar k\setminus w^{-1}(\infty)$,
let $e$ be the ramification index of $w$ at $\alpha$
and $\nu$ the order of the derivative $w'$ at $\alpha$.
Then $e=\nu+1$ if the ramification at $\alpha$ is tame,
and $e\leq\nu$ otherwise.
In particular, $\alpha$ is a critical point of $w$ iff $\alpha$ is a root of $f'c-fc'$.
\item[(ii)] Assume additionally that $\deg f>\deg c$, $f'\neq0$, and that $c$ is separable. 
Then the finite critical values of $w$ are the roots of $\Disc_X(f-Uc)$, more precisely
$$
\Disc_X(f(X)-Uc(X))=a\cdot\prod_{(f'c-fc')(\alpha)=0}(U-w(\alpha)),
$$
with $a\in k^\times$ and the product over the roots of $f'c-fc'$ in $\bar k$ with multiplicity.
\end{enumerate}
\end{lemma}

\begin{proof}
{\bf (i).} We identify $U=w(X)$, viewing $U$ as an element of $k(X)$. This gives the extension of function fields $k(X)/k(U)$ corresponding to the rational map $w$. Let $\alpha\in\bar k$ be such that $c(\alpha)\neq 0$ and denote $\beta=w(\alpha)$. Then $X-\alpha$ is a local parameter for $\P^1_X$ at the point $\alpha$ and $U-\beta$ is a local parameter for $\P^1_U$ at the point $\beta$. 
We have $U-\beta=w(X)-\beta=(X-\alpha)^eu(X)$ 
with $u\in k(X)$ which has no pole at $\alpha$. 
Taking derivatives we obtain
$$
 w'(X)=e(X-\alpha)^{e-1}u(X)+(X-\alpha)^eu'(X).
$$ 
As $\alpha$ is neither a pole nor a zero of $u$, and hence also not a pole of $u'$,
the multiplicity $\nu$ of the zero $\alpha$ of $w'$ 
is $\nu=e-1$ if ${\rm char}(k)\nmid e$, and $\nu\geq e$ if ${\rm char}(k)\mid e$.

{\bf (ii).} Denote $h=(f',c')\in\bar k[X]$ and write $f'=hu,c'=hv,u,v\in\bar k[X]$. 
As $f'\neq 0$, also $u\neq 0$.
If $\deg c=0$ the 
assertion of (ii) is immediate from (i) and (\ref{eq:disc2}), so let us assume $\deg c>0$ and then 
$c',v\neq 0$ by our assumption that $c$ is separable. Using the properties of resultants and 
discriminants (\ref{eq:resdefinition}),(\ref{eq:disc}),(\ref{eq:resbimult}),(\ref{eq:ressym}),(\ref{eq:resalt}) and using $\sim$ to denote
equality up to a non-zero multiplicative constant (in $k$),
we calculate:
\begin{eqnarray*}
\Disc_X(f-Uc)&\stackrel{(\ref{eq:disc})}{\sim}&\Res(f'-Uc',f-Uc)\\
&\stackrel{(\ref{eq:resbimult})}=& \Res(h,f-Uc)\cdot\Res(u-Uv,f-Uc)\\
&\stackrel{(\ref{eq:resbimult})}=&\Res(h,f-Uc)\cdot\frac{\Res(u-Uv,fv-Ucv)}{\Res(u-Uv,v)}\\
&\stackrel{(\ref{eq:ressym}),(\ref{eq:resalt})}\sim&\Res(h,f-Uc)\cdot\frac{\Res(u-Uv,fv-uc)}{\Res(u,v)}\\
&\stackrel{(\ref{eq:ressym})}\sim&\Res(h,f-Uc)\cdot\Res(uc-fv,u-Uv)\\
&\stackrel{(\ref{eq:resdefinition})}\sim&\prod_{h(\alpha)=0}\left(c(\alpha)U-f(\alpha)\right)\prod_{(uc-fv)(\alpha)=0}(u(\alpha)-Uv(\alpha))\\
&\sim&\prod_{h(\alpha)=0}\left(U-\frac{f(\alpha)}{c(\alpha)}\right)\prod_{(uc-fv)(\alpha)=0}\left(U-\frac{f'(\alpha)}{c'(\alpha)}\right)\\
&\sim&\prod_{(f'c-fc')(\alpha)=0}(U-w(\alpha)).
\end{eqnarray*}
In the last line we used the fact that for a root $\alpha$ of $(f'c-fc')/h$ we have $f'(\alpha)/c'(\alpha)=w(\alpha)$.
\end{proof}

\subsection{Monodromy of rational functions}\label{sec:monodromy}
The definitions and results in this subsection will be used only in Section \ref{sec:two}.
Let $k$ be a field of characteristic $p\geq0$. 
The {\em (arithmetic) monodromy} $\Mon(w)$ of a non-constant rational function $w\in k(X)$ 
for which the extension $k(X)/k(w)$ is separable
is the Galois group of the Galois closure of $k(X)/k(w)$.
When ambiguity about the base field may arise we write $\Mon_k(w)$. 
We usually view $\Mon(w)$ as a subgroup of $S_n$ with $n=\deg w$ via its action on the generic fiber of the corresponding morphism $w\colon \P^1\to\P^1$.
A rational function $w\in k(X)$ is called \emph{indecomposable} if it cannot be written as a composition $w=u\circ v$ with $u,v\in
k(X)$ where $\deg u,\deg v>1$.

\begin{lemma}\label{lem:indecomposable_primitive}
Let $w\in k(X)$ with $k(X)/k(w)$ separable.
Then $w$ is indecomposable if and only if
$\Mon(w)\leqslant S_n$ is primitive.
\end{lemma}

\begin{proof}
If $w=u\circ v$ with $\deg u>1$, $\deg v>1$, 
and $\alpha_1,\dots,\alpha_m\in\overline{k(T)}$ are the roots of $u-T$,
then $v^{-1}(\alpha_1),\dots,v^{-1}(\alpha_m)$ is a nontrivial partition of the roots of $w-T$,
which is preserved by $\Mon(w)$.
The converse implication is proven for example in \cite[Lemma 2]{Fried} in the case of polynomials,
and the case of rational functions can be proven similarly using L\"uroth's theorem.
\end{proof}

\begin{definition}\label{def:ramification type}
Let $w\colon \P^1\to\P^1$ be a tamely ramified rational function defined over an algebraically closed field $k$. 
The \emph{ramification type} of $w$ (called \emph{passport} in \cite{AdZv15}) is the unordered tuple of partitions $(\lambda_1,\ldots,\lambda_r)$, where $b_1,\ldots,b_r$ are the distinct branch points of $w$ and $\lambda_i$ is the partition of $n=\deg w$ given by the ramification indices over $b_i$.
We write partitions of $n$ as $n_1^{e_1}\dots n_r^{e_r}$ with $0\leq n_1<\dots<n_r$ and $\sum_{i=1}^re_in_i=n$.
\end{definition} 
We recall that the inertia subgroup of $\Mon(w)\leqslant S_n$ corresponding to a branch point $b_i$ of 
the tamely ramified rational function $w$ is a cyclic subgroup of $S_n$ with generator of cycle type
precisely the partition of $n$ given by the ramification indices over $b_i$.
Moreover, $\Mon(w)$ is generated by the inertia subgroups.

\begin{lemma}\label{lem:sga} Let $w\colon\P_k^1\to\P_k^1$ be a tamely ramified rational function defined over an algebraically closed field $k$ with monodromy group $G\leqslant S_n$ and ramification type $\Pi=(\lambda_1,\ldots,\lambda_r)$. Then there exists a rational function $W\colon\P^1_\C\to\P^1_\C$ with the same monodromy and ramification type.\end{lemma}

\begin{proof} This is essentially a direct consequence of the comparison theorem for the \'etale 
fundamental group \cite[\S XIII, Corollaire 2.12]{GrRa71}. We use the terminology and notation of 
\cite[\S XIII]{GrRa71}. Let $b_1,\ldots,b_r\in\P^1(k)$ be the branch points of $w$ and denote $U=
\P^1\setminus\{b_1,\ldots,b_r\}$. By \cite[\S XIII, Corollaire 2.12]{GrRa71} there exist elements 
$x_i\in\pi_1^t(U),1\le i\le r$ such that $x_1,\ldots,x_r$ generate $\pi_1^t(U)$, $\prod_{i=1}^rx_i=1$ and each $x_i$ generates a tame inertia group at $b_i$. 

The map $w$ induces a finite \'etale cover $V=w^{-1}(U)\to U$ and therefore (using the assumption that $w$ is tamely ramified) a surjection $\delta\colon\pi_1^t(U)\twoheadrightarrow G$. Denote $\sigma_i=\delta(x_i)$. The elements $\sigma_1,\ldots,\sigma_r$ generate $G$, satisfy $\prod_{i=1}^r\sigma_i=1$ and each $\sigma_i\in G\leqslant S_n$ has cycle structure $\lambda_i$. Now by the Riemann Existence Theorem (see e.g. \cite[Corollary 5.11]{vol96}) there exists a finite irreducible cover $W\colon Y\to \P^1_\mathbb C$ with monodromy group $G$ and ramification type $\Pi$. The genus of $Y$ can be determined by the Riemann-Hurwitz formula from its ramification type and is therefore 0 (since the same calculation gives the genus of the domain of $w$, since $w$ is tamely ramified).
\end{proof}

\begin{lemma}\label{lem:primitive}
Let $w\colon\mathbb{P}^1\rightarrow\mathbb{P}^1$ 
be a rational function of degree $n>1$ over the algebraically closed field $k$
with monodromy group $\Mon(w)\leqslant S_n$.
Assume that $w$ has at most two critical values $a,b$,
and that $w^{-1}(a)$ contains only one critical point $a'$.
If the ramification index of $a'$ is prime, and all ramification over $b$ is 
tame, then $\Mon(w)$ is primitive.
\end{lemma}

\begin{proof}
Let $L$ be the Galois closure of the extension $k(X)/k(T)$ induced by $w$.
Let $k(T)\subseteq M\subsetneq k(X)$ be an intermediate extension. 
Since there is only one ramified prime of $k(X)$ over $T=a$ and it has prime ramification index, 
the extension $M/k(T)$ has to be unramified over $T=a$. 
However, it is also tamely ramified over $T=b$ and unramified elsewhere, 
hence $M=k(T)$ by the Riemann-Hurwitz formula.
So $k(X)/k(T)$ is a minimal proper extension,
hence by Galois theory $\Gal(L/k(X))$ is maximal in $\Gal(L/k(T))$,
which means that $\Mon(w)=\Gal(L/k(T))\leqslant S_n$ is primitive.
\end{proof}

\section{Abelian groups}
\label{sec:abelian}
\label{sec:generators}

\noindent
In this section we first prove the group-theoretic fact mentioned in Remark \ref{rem:joachim} 
and then verify Conjecture \ref{conj} for abelian groups. 

If $G$ is a group and $a\in G$ is an element we denote by $a^G$ the conjugacy class of $a$. For a subset $S\subseteq G$ we denote $S^G=\bigcup_{a\in S}a^G$.
Recall that for a finite group $G$ we denote by $p(G)$ the subgroup generated by its $p$-Sylow subgroups.
 A \emph{cyclic-by-$p$} group is an extension of a cyclic group by a $p$-group. 

\begin{lemma}\label{lem:gen_by_cyc} 
Let $G$ be a finite group. There exist $x_1,\ldots,x_r\in G$ with $r=\max\{d(G^\ab),1\}$ such that $G=\langle x_1^G,\ldots,x_r^G\rangle$.\end{lemma}

\begin{proof} 
This is a special case of the main theorem in \cite{Kut76},
see also \cite{Bae64}.
\end{proof}

\begin{proposition}\label{prop:gen_by_cyc_p} 
Let $G$ be a finite group and $p$ a prime number. There exist subgroups $Q_1,\ldots,Q_r$ of $G$ such that $r=\max\{d((G/p(G))^\ab),1\}$, each $Q_i$ is cyclic-by-$p$ and $G=\langle Q_1^G,\ldots,Q_r^G\rangle$.
\end{proposition}

\begin{proof} 
Let $H=G/p(G)$, so that $r=\max\{d(H^\ab),1\}$. 
By Lemma~\ref{lem:gen_by_cyc} there exist $y_1,\ldots,y_r\in H$ with $H=\langle y_1^H,\ldots,y_r^H\rangle$. 
Let $P\leqslant G$ be a $p$-Sylow subgroup. If $P=1$, i.e. $(|G|,p)=1$, then the claim holds with $Q_i=\langle y_i\rangle$, so we may assume $P\neq 1$. By Frattini's argument \cite[Theorem 3.7]{Gor80} we have $p(G)N_G(P)=G$, so we can choose $x_1,\ldots,x_r\in N_G(P)$ such that $y_i = p(G)x_i$.
Denote $Q_i=\langle x_i\rangle P$. Each $Q_i$ is a subgroup of $G$ since $x_i\in N_G(P)$, it is cyclic-by-$p$ since it is an extension of $\langle x_i\rangle/(\langle x_i\rangle\cap P)$ by $P$, and finally since
$p(G)=\langle P^G\rangle$ and $G=\langle p(G),x_1^G,\ldots,x_r^G\rangle$ we have $G=\langle P^G,x_1^G,\ldots,x_r^G\rangle=\langle Q_1^G,\ldots Q_r^G\rangle$, as required.
\end{proof}

\begin{theorem}\label{thm:abelian}
Conjecture \ref{conj} holds for $G$ abelian.
\end{theorem}

\begin{proof}
	Let $q=p^\nu$ and let $G$ be a nontrivial finite abelian group.
	The claim is that $\ram_{\mathbb{F}_q(T)}(G)=1$ if $G$ is a $p$-group and $\ram_{\mathbb{F}_q(T)}(G)=d(G/p(G))$ otherwise.
	
	We first prove the lower bound.
	Let $L/\FF_q(T)$ be an abelian Galois extension with Galois group $G$ and $\FF_q$ algebraically closed in $L$.
	If $G$ is a $p$-group, the lower bound $\ram_{\mathbb{F}_q(T)}(G)\geq 1$ follows from the Riemann-Hurwitz formula.
   If $G$ is not a $p$-group, then $G/p(G)$ is nontrivial, 
   and since $\ram_{\mathbb{F}_q(T)}(G)\geq\ram_{\mathbb{F}_q(T)}(G/p(G))$ we can assume that $p(G)=1$. 
    Let $P_1,\ldots, P_k$ be all the finite primes of $\mathbb{F}_q(T)$ ramified in $L$, and let $I_1,\ldots, I_k$ be the corresponding inertia groups (note that $I_i$ are well defined since the extension is abelian). Then each $I_i$ is cyclic, as $P_i$ is tamely ramified in $L$. 
As also the infinite prime of $\FF_q(T)$ is tamely ramified in $L$,
$G=\left< I_1,\ldots, I_k\right>$, cf.~\cite[Proposition 4.4.6]{Serre}.
In particular, $\ram_{\mathbb{F}_q(T)}(G)\geq k\geq d(G)$.

	To show that the lower bound is tight, we construct a suitable Galois extension. 
	Write $G= G/p(G) \times p(G)$, let $k$ and $p^\alpha$ be the exponents of $G/p(G)$ respectively $p(G)$,
	and let $d=\max\{d(G/p(G)),1\}$. 
	Choose distinct primes $P_1, \ldots, P_d$ such that $k(q-1)\mid q^{\deg P_i}-1$ (i.e.\ such that the multiplicative order of $q$ modulo $k(q-1)$ divides $\deg P_i$). 
	Let $R_i$ be the completion of $\FF_q[T]$ at $P_i$ and $E_i$ its fraction field.  
	By \cite[Proposition~II.5.7]{Neukirch}, $E_i^\times = \left<\pi\right> \times(\mathbb{F}_{q^{\deg P_i}})^\times \times G_i$, where 
	$\pi$ is a uniformizer and
	\[
		G_i = 1+ \pi\FF_{q^{\deg P_{i}}}[[\pi]] \cong \ZZ_p^{\mathbb{N}}.
	\]
In particular,  $R_i^\times = (\mathbb{F}_{q^{\deg P_i}})^\times \times G_i$. 
	By class field theory (take inverse limit in \cite[Proposition~2.2, Theorem~2.3, and Theorem~3.2]{Hayes}) the maximal abelian extension of $\FF_q(T)$ which is regular over $\FF_q$, tamely ramified at infinity, and unramified outside $\{P_i,\infty\}$ has Galois group $R_i^\times$ and the tame inertia group at $\infty$ corresponds to $\mathbb{F}_{q}^\times$. Hence $R_i^\times/\mathbb{F}_q^\times$ is the Galois group of the maximal abelian extension that is regular over $\FF_q$ and unramified outside $\{P_i\}$. As $\mathbb{F}_q(T)$ has no unramified extensions regular over $\mathbb{F}_q$, those extensions for different $i$ are linearly disjoint, so the Galois group of the maximal abelian extension that is regular over $\FF_q$ and unramified outside $\{P_1,\ldots, P_d\}$ is isomorphic to 
	\[
		\prod_{i=1}^d R_i^\times/\mathbb{F}_q^\times \cong \prod_{i=1}^d C_{(q^{\deg P_i}-1)/(q-1)} \times \ZZ_p^{\mathbb{N}}.
	\]
	By the assumption that $k\mid (q^{\deg P_i}-1)/(q-1)$, we get that $G/p(G)$ is a quotient of $ \prod_i C_{(q^{\deg P_i}-1)/(q-1)} $ (this is vacantly true if $G/p(G)=1$). Obviously, $p(G)$ is a quotient of $\ZZ_p^{\mathbb{N}}$. Hence $G=G/p(G)\times p(G)$ is a quotient of $\prod_i R_i^\times/\mathbb{F}_q^\times$. This implies we can realize $G$ as the Galois group of a geometric Galois extension of $\mathbb{F}_q(T)$ unramified outside $\{P_1,\ldots, P_d\}$, as needed.
\end{proof}

\begin{corollary}\label{cor:lowerbound}
For every prime power $q=p^\nu$ and every nontrivial finite group $G$,
$$
 \ram_{\mathbb{F}_{q}(T)}(G) \geq \max\left\{d((G/p(G))^{\rm ab}),1\right\}.
$$
\end{corollary}

\begin{remark}
For $K=\mathbb{F}_p(T)$, a conjecture 
was posed in \cite[Conjecture 1.1]{DeWitt} and claimed to be proven
for abelian finite groups,
namely that
$$
 \ram_{\mathbb{F}_p(T)}(G) = \begin{cases} d((G/p(G))^{\rm ab})+1,&\mbox{if }p\,|\,|G^{\rm ab}|\\\max\{d((G/p(G))^{\rm ab}),1\},
 &\mbox{otherwise}\end{cases}.
$$
%
However, already 
in the case $G=\mathbb{Z}/2p\mathbb{Z}$, $p>2$,
the number obtained there is too big:
For example, $G=\mathbb{Z}/6\mathbb{Z}$ can be realized over $\mathbb{F}_3(T)$ with only one ramified prime of degree $2$,
for example by $\mathbb{F}_3(T,x,y)$ with $x^2=y^3-y=(T^2+1)^{-1}$.
The incorrect lower bound in the case $p||G^{\rm ab}|$ in \cite{DeWitt} 
presumably comes from the false assumption there that an extension with such group $G$ must be wildly ramified at infinity. 

In any case, at least for $q=p$, \cite[Theorem 2.6]{DeWitt} proves
Conjecture \ref{conj} for abelian groups of order prime to $p$
(and thus proves a special case of Theorem~\ref{thm:abelian}).
Moreover, 
\cite[Theorem 3.6]{DeWitt} 
proves Conjecture \ref{conj}
for all semiabelian $\ell$-groups $G$, with $\ell$ a prime number different from $p$
(building on results from \cite{KisilevskySonn} and \cite{KisilevskyNeftinSonn}),
and \cite[Corollary 6.4]{DeWitt} proves Conjecture \ref{conj}
for all $\ell$-groups with $\ell$ a prime number dividing $q-1$
(both results are stated for $q=p$, but the proofs go through also for general $q$).

We note however that the proof of \cite[Theorem 6.9]{DeWitt}, which claims to prove the above conjecture for all nilpotent groups,
seems to be incomplete regarding  both  $p$-groups (\cite[Theorem 2.5]{DeWitt}) 
and $\ell$-groups with $\ell$ coprime to $p-1$ (\cite[Corollary 6.8]{DeWitt}).
\end{remark}

\section{$S_n$ and $A_n$ for $p>n$}

\label{sec:tworamified}

\noindent
In this section we obtain Galois extensions of group $S_n$ and $A_n$ with ramification over at most two primes
in the case where $n$ is smaller than the characteristic $p$.
We start with splitting fields of polynomials
of the form $f(X)-T$ in Section \ref{sec:Morse}
and $f(X)-Tc(X)$ in Section \ref{sec:two},
which produce extensions with one ramified finite prime 
and tame ramification over infinity.
In Section \ref{sec:oneramified}
we then explain how to eliminate the tame ramification at infinity
in many cases.
Finally, in Section \ref{sec:twinprimes}
we present yet another approach to obtain extensions with two ramified primes,
by working with trinomials and applying recent results on small gaps between primes in function fields.

We recall that 
if $f\in\F_q[T,X]$ is monic in $X$,
then the only primes of $K=\mathbb{F}_q(T)$ that possibly ramify in 
the splitting field $L$ of $f$ over $K$
are the divisors of the discriminant $\Disc_X(f)$, and possibly the infinite prime of $K$ (Lemma~\ref{lem:Dedekind}).
In particular, if $\Disc_X(f)$ is irreducible or, more generally, a prime power, at most two primes of $K$ ramify in $L$.

\subsection{Two ramified primes via Morse polynomials}
\label{sec:Morse}

We start with splitting fields of polynomials of the simple from $f(X)-T$,
where in some cases the classical theory of Morse polynomials
produces suitable extensions of group $S_n$ with two ramified primes.
Let $k$ be a field of characteristic $p\neq2$,
and recall that
a polynomial $f\in k[X]$ of degree $n$ not divisible by $p$ is {\em Morse} if it has exactly $n-1$ distinct critical values, i.e.
the roots $\alpha_1,\dots,\alpha_{n-1}$ of $f'$ in $\bar{k}$ are simple and
$f(\alpha_i)\neq f(\alpha_j)$ for $i\neq j$.
If $f$ is Morse, then
${\rm Gal}(f(X)-T/\overline{k}(T))=S_n$
by \cite[Theorem 4.4.5]{Serre},
so the splitting field of $f(X)-T$ over $k(T)$
is geometric with Galois group $S_n$.

\begin{lemma}\label{lem:f'Morse_disc_irred}
Let $f\in k[X]$ be Morse of degree $n$ with $p\nmid n$. 
Then $f'$ is irreducible if and only if $D(T):=\disc_X(f(X)-T)\in k[T]$ is irreducible.
\end{lemma}

\begin{proof}
Let $A$ be the set of roots of $f'$, and $G={\rm Gal}(k(A)/k)$.
By (\ref{eq:disc2}) or Lemma~\ref{lem:crit}, $D\sim\prod_{a\in A}(T-f(a))$
(where as before $\sim$ denotes equality up to a non-zero constant).
The fact that $f$ is Morse implies that $f$ is injective on $A$,
so $f$ induces an isomorphism of $G$-sets $A\rightarrow f(A)$.
In particular, $k(f(a))=k(a)$ for every $a\in A$,
hence $D\sim\prod_{a\in A}(T-f(a))$ is irreducible
if and only if $f'\sim\prod_{a\in A}(X-a)$ is.
\end{proof}

\begin{proposition}\label{thm:S_nlargeq}\label{thm:large_q}
    Let $n\geq3$ and $q=p^\nu$ with $p>n$.
    There are $\frac{q^n}{n-1} + O_n(q^{n-1})$ many 
    monic Morse $f\in\mathbb{F}_q[X]$ of degree $n$
    such that the splitting field of $f(X)-T$ over $\mathbb{F}_q(T)$ is geometric with Galois group $S_n$,
    and $\disc_X(f(X)-T)$ is irreducible.
\end{proposition}

\begin{proof}
Let $M_n\cong \AA^n$ be the space of monic polynomials of degree $n$ (identified with the $n$-tuple of non-leading coefficients) considered as a variety over $\mathbb{F}_q$.
The subset $\mathcal{M}\subseteq M_n$ of monic Morse polynomials of degree $n$
is Zariski-open and dense 
with complement $M_n\setminus\mathcal{M}$ the zero set of a polynomial in the coefficients $a_0,\dots,a_{n-1}$
of degree $O_n(1)$, see \cite[Proposition 4.3 and (2) in its proof]{Geyer}.
Thus 
the elementary estimate \cite[Lemma 1]{LangWeil} gives that
the set $\mathcal{M}(\F_q)$ of $f\in M_n(\mathbb{F}_q)$ that are Morse
satisfies
$|\mathcal{M}(\F_q)|= q^n+O_n(q^{n-1})$.

Let $\mathcal{P}$ denote the set of $f\in M_n(\mathbb{F}_q)$
for which $f'$ is irreducible.
As $p>n$, the map $M_n(\mathbb{F}_q)\rightarrow M_{n-1}(\mathbb{F}_q)$, $f\mapsto \frac{1}{n}f'$
is surjective with fibers of size $q$,
so by the Prime Polynomial Theorem,
$|\mathcal{P}|=\frac{q^n}{n-1}+O_n(q^{(n-1)/2})$.
It follows that $$|\mathcal{M}(\F_q)\cap\mathcal{P}|=|\mathcal P\setminus(M_n(\F_q)\setminus \mathcal M(\F_q))|=\frac{q^n}{n-1} + O_n(q^{n-1}),$$
since by the above $|M_n(\F_q)\setminus\mathcal M(\F_q)|=O_n(q^{n-1})$.
Now for every $f\in\mathcal{M}(\F_q)$,
the splitting field of $f-T$ is geometric with Galois group $S_n$ (see above),
and
$\disc_X(f-T)$ is irreducible
if and only if  $f\in\mathcal{P}$  (Lemma~\ref{lem:f'Morse_disc_irred}),
so the claim follows.
\end{proof}

\begin{lemma}\label{lem:f'irred_Morse}
Let $f\in\mathbb{F}_q[X]$ of degree $n$ where $q=p^\nu$, $p\nmid n$ and $n-1$ is prime.
If $f'$ is irreducible, then $f$ is Morse.
\end{lemma}

\begin{proof}
Suppose that $f'$ is irreducible
and that there exist $a\neq b$ in $\overline{\mathbb{F}}_q$ with $f'(a)=f'(b)=0$ and $f(a)=f(b)$. 
As $f'$ is irreducible of degree $n-1$, we have $a,b\in\mathbb{F}_{q^{n-1}}$
and there exists $1\neq\sigma\in{\rm Gal}(\mathbb{F}_{q^{n-1}}/\mathbb{F}_q)$ with $b=a^\sigma$.
Thus $c:=f(a)=f(b)=f(a)^\sigma$ is in the fixed field of $\sigma$, which is $\mathbb{F}_q$ due to the assumption that
$n-1$ is prime.
So $f-c\in\mathbb{F}_q[X]$ has root $a$, hence $f-c=f'g$
for some $g\in\mathbb{F}_q[X]$ of degree $1$.
Deriving gives 
$f''g=f'(1-g')$, 
so since $f'$ is irreducible and ${\rm deg}(f')>{\rm deg}(g)$,
we conclude that $f'|f''$, which contradicts the separability of $f'$.
\end{proof}

\begin{proposition}\label{prop:Morse2}
Let $p\geq3$ and $q=p^\nu$. Suppose that $n-1$ is prime and $n\in\{2,\dots,p-1\}\cup\{p+1\}$.
Then there exists $f\in\mathbb{F}_q[X]$ of degree $n$ such that
the splitting field of $f(X)-T$ over $\mathbb{F}_q(T)$ is geometric with Galois group $S_n$,
and $\disc_X(f(X)-T)$ is irreducible.
\end{proposition}

\begin{proof}
There exists $f\in\mathbb{F}_q[X]$ of degree $n$ with $f'$ irreducible:
If $n<p$ we can choose $f'$ to be any irreducible polynomial of degree $n-1$;
if $n=p+1$ we can choose $f'$ to be the minimal polynomial of any $\alpha\in\mathbb{F}_{q^p}\setminus\mathbb{F}_q$
with ${\rm Tr}_{\mathbb{F}_{q^p}/\mathbb{F}_q}(\alpha)=0$, of which there exist $q^{p-1}-q$ many.
By Lemma~\ref{lem:f'irred_Morse}, $f$ is Morse, 
and therefore 
the splitting field of $f-T$ is geometric with Galois group $S_n$,
and $\disc_X(f-T)$ is irreducible by Lemma~\ref{lem:f'Morse_disc_irred}.
\end{proof}

\subsection{Two ramified primes via monodromy of rational functions}
\label{sec:two}

The main result in this subsection 
is Theorem~\ref{thm:tame2} below,
which generalizes the ideas of the previous subsection by studying polynomials of the form $f(X)-Tc(X)$ rather than $f(X)-T$.
We adopt the notation and terminology of Section \ref{sec:monodromy}.

The proof of the following key proposition uses the classification of monodromy of indecomposable rational functions (in special cases) by Adrianov and Zvonkin \cite{AdZv15, Adr17}, which is based on the classification of primitive permutation groups containing a cycle by Jones \cite{Jon14} combined with the earlier monodromy classification results of M\"uller \cite{Mul93, Mul94_}. These results rely heavily on the CFSG, although in many cases more elementary group theory is sufficient.
Recall the definition of the ramification type of a rational function (Definition~\ref{def:ramification type}).

\begin{proposition}\label{prop:mon}

Let $k$ be algebraically closed of characteristic $p\geq0$. 
Let $f,c\in k[X]$ be polynomials with $\deg f=n,\deg c=m\le n-2,(f,c)=1$ and $c$ squarefree.
 Assume further that $w=\frac{f}{c}\in k(X)$ is indecomposable and $m\not\equiv n\pmod p$ if $p>0$.
\begin{enumerate}
\item[(i)] Assume that $p\neq 2$, $m\equiv n\pmod 2$ and $g=f'c-fc'$ is squarefree. 
Then $\Mon(w)=S_n$ or one of the following:
\begin{itemize}
\item $n=6,m=2,\Mon(w)=PGL_2(5)$ with ram.~type $(1^22^2,1^22^2,2^3,1^24^1)$.
\item $n=8,m=2,\Mon(w)=PGL_2(7)$ with ram.~type $(1^22^3,1^22^3,1^22^3,1^26^1)$.
\item $n=9,m=1,\Mon(w)=AGL_2(3)$ with ram.~type $(1^32^3,1^32^3,1^32^3,1^18^1)$.
\item $n=10,m=2,\Mon(w)=P\Gamma L_2(9)$ with ram.~type $(1^42^3,1^42^3,2^5,1^28^1)$.
\end{itemize}
\item[(ii)] Assume that $p\notin\{2,3\}$, $m\not\equiv n\pmod 2$, and $f'c-fc'=g^2$ where $g\in k[X]$ is squarefree. 
Then $\Mon(w)=A_n$ or one of the following:
\begin{itemize}
\item $n=8,m=1,\Mon(w)=L_2(7)$ or $A\Gamma L_1(8)$ with ram.~type $(1^23^2,1^23^2,1^17^1)$.
\item $n=9,m=0,\Mon(w)=P\Gamma L_2(8)$ with ram.~type $(1^33^2,1^33^2,9^1)$.
\item $n=12,m=1,\Mon(w)=M_{12}$ with ram.~type $(1^33^3,1^33^3,1^111^1)$.
\item $n=24,m=1,\Mon(w)=M_{24}$ with ram.~type $(1^63^6,1^63^6,1^123^1)$.
\end{itemize}
\end{enumerate}
\end{proposition}

\begin{proof} 

The classification of monodromy groups of indecomposable rational functions with precisely one multiple pole over $\C$ is given in \cite[\S 2.2 and Theorem 12]{AdZv15} (case of 3 critical values) and \cite[Theorem 1]{Adr17} (case of 4 or more critical values). 
By Lemma~\ref{lem:sga}  the possible monodromy and ramification type pairs for a tamely ramified rational function over an algebraically closed field in arbitrary characteristic can only be the ones occurring over $\C$. 
The lists of possible monodromy groups different from $A_n,S_n,C_n,D_n$  together with the corresponding ramification data appear in \cite[\S 3]{AdZv15} and \cite[Table 1]{Adr17}. We will refer to these henceforth as \emph{the tables}.

The case $C_n$ only occurs if $n$ is prime with ram.~type $(n^1,n^1)$,
and the case $D_n$ occurs only when $n>2$ is prime with ram.~type 
$(1^12^{(n-1)/2}, 1^12^{(n-1)/2}, n^1)$, see \cite[Section 2.2]{AdZv15}
where this can be read off from the corresponding dessins d'enfants. 
In both cases there is a totally ramified branch point (i.e.\ after applying a fractional-linear transformation $w$ becomes a polynomial of degree $n$).

Consider the rational map $w=\frac fc\colon\P^1\to\P^1$. 
By our assumptions it has degree $n$, simple poles at the roots of $c$ and a pole of multiplicity $n-m$ at infinity. 
Note that $\deg(f'c-fc')=n+m-1$ as $m\not\equiv n\pmod p$. 
\\

$(i)$.
Assume that $p\neq 2$, $m\equiv n\pmod 2$ and $g=f'c-fc'$ is squarefree.
Assume further that $\Mon(w)\neq S_n$, in particular $n>2$. 
As all zeros of $g$ are simple and $p\neq 2$, 
the ramification points of $w$ are
the zeros of $g$ with ramification index 2 (Lemma~\ref{lem:crit}) 
and $\infty$ with ramification index $n-m$. 
In particular, $w$ is tamely ramified (as $p\neq 2$ and $p\nmid n-m$)
and the ramification type of $w$ has the form 
$$
 \left(1^{a_1}2^{b_1},\ldots,1^{a_r}2^{b_r},1^{m}(n-m)^1\right), 
$$
where $r$ is the number of finite critical values of $w$
and $a_i+2b_i=n$ for all $i$. 
As $\sum_{i=1}^rb_i=\deg g\geq n-1>n/2$, we see that $r\geq2$.

Note that $\Mon(w)$ cannot be $A_n$ since the last entry in the ramification type corresponds to an odd permutation, 
it cannot be $C_n$ with $n$ prime since then it would have only one finite critical value and it cannot be $D_n$ with $n$ prime since this case occurs only if $w$ has a totally ramified branch point (i.e. $n^1$ in its ramification type), but if $n>2$ is prime then $m$ must be odd and thus $w$ is not a polynomial. 
Therefore $\Mon(w)$ must be one of the entries in the tables.

Going through the tables we see that the only entries with a ramification type as above (with $n-m$ even) are:
\begin{itemize}
\item {\bf 4/6.2} with $\Mon(w)=PGL_2(5)$ and ramification type $(1^22^2,1^22^2,2^3,1^24^1)$.
\item {\bf 4/8.4} with $\Mon(w)=PGL_2(7)$ and ramification type $(1^22^3,1^22^3,1^22^3,1^26^1)$.
\item {\bf 4/9.1} with $\Mon(w)=AGL_2(3)$ and ramification type $(1^32^3,1^32^3,1^32^3,1^18^1)$.
\item {\bf 4/10.1} with $\Mon(w)=P\Gamma L_2(9)$ and ramification type $(1^42^3,1^42^3,2^5,1^28^1)$.
\end{itemize}

$(ii)$.
Now assume instead that $p\notin\{2,3\}$, $m\not\equiv n\pmod 2$ and $f'c-fc'=g^2$ with $g$ squarefree.
As all zeros of $g^2$ have multiplicity $2$ and $p\notin\{2,3\}$, 
the ramification points of $w$ are
the zeros of $g$ with ramification index 3 (Lemma~\ref{lem:crit}) 
and $\infty$ with ramification index $n-m$. 
In particular, $w$ is tamely ramified (as $p\neq 3$ and $p\nmid n-m$) and has a ramification type of the form
$$
 \left(1^{a_1}3^{b_1},\ldots,1^{a_r}3^{b_r},1^{m}(n-m)^1\right), 
$$
where $r$ is the number of finite critical values of $w$. 
Since now $n-m$ is odd, all entries in the ramification type correspond to even permutations.
So since $\Mon(w)$ is generated by the inertia groups, $\Mon(w)\leqslant A_n$, which excludes the case $S_n$. 
If $n=3$ (and $m=0$), then $r=1$ and $\Mon(w)=A_3=C_3$.
If $n>3$, then similarly to part (i) one can argue that $r\geq2$
and that one can exclude the cases $C_n$ and $D_n$ ($D_n$ comes with ramification index 2 at the finite critical points).

Now once again we go over the tables and list the entries with a ramification type of the above form (with $n-m$ odd):
\begin{itemize}
\item {\bf 8.1} with $\Mon(w)=A\Gamma L_1(8)$ and ramification type $(1^23^2,1^23^2,1^17^1)$.
\item {\bf 8.9} with $\Mon(w)=L_2(7)$ and ramification type $(1^23^2,1^23^2,1^17^1)$.
\item {\bf 9.7} with $\Mon(w)=P\Gamma L_2(8)$ and ramification type $(1^33^2,1^33^2,9^1)$.
\item {\bf 12.10} with $\Mon(w)=M_{12}$ and ramification type $(1^33^3,1^33^3,1^111^1)$.
\item {\bf 24.5} with $\Mon(w)=M_{24}$ and ramification type $(1^63^6,1^63^6,1^123^1)$.\qedhere
\end{itemize}
\end{proof}

\begin{remark} \label{remark:manning}
While the above proposition relies on the CFSG in general, 
if we assume that $n-m$ is small, much more elementary group theory is sufficient. 
For example if $n-m\le 15$ the result can be derived from the classification by Manning of primitive permutation groups of class $\le 15$ (i.e. primitive groups which have a non-identity element fixing all but $\le 15$ points). See \cite[\S II.15]{Wie64} for the list of references.
\end{remark}

\begin{lemma}\label{lem:indec} 
Let $w\colon\P^1\to\P^1$ be a tamely ramified rational function defined over a field $k$. 
Assume that $w$ has a multiple pole at $\infty$  and all other poles over $\overline{k}$ are simple. Assume further that the numerator $g\in k[X]$ of the derivative $w'$ is irreducible or the square of an irreducible polynomial.
Then $w$ is indecomposable over $k$.
\end{lemma}

\begin{proof} It will be convenient to introduce two new variable symbols $T,U$ and consider three separate copies of $\P^1_k$ which we will denote $\P^1_X,\P^1_T,\P^1_U$. 
For the sake of contradiction, 
assume that $w=v\circ u$ is the composition of two rational functions
$u,v\in k(X)$ with $\deg u,\deg v>1$. 
Since $w(\infty)=\infty$, assume without loss of generality that $u(\infty)=\infty$ and $v(\infty)=\infty$.
Consider the corresponding coverings
$$
 \P^1_X\xrightarrow{u}\P^1_T\xrightarrow{v}\P^1_U.
$$
For a divisor $D=\sum_P n_PP$ on $\P^1_X$ we denote 
$$
 D^0=\sum_{w(P)\neq\infty}n_PP,\quad D^\infty=\sum_{w(P)=\infty}n_PP
$$ 
and similarly for a divisor $D$ on $\P^1_T$ we denote
$$
 D^0=\sum_{v(P)\neq\infty}n_PP,\quad D^\infty=\sum_{v(P)=\infty}n_PP.
$$
Consider the differents $\mathfrak d_w$, $\mathfrak{d}_u$, $\mathfrak{d}_v$ of $w$, $u$ and $v$
as divisors on $\P^1_X$, $\P^1_X$ and $\P^1_T$, respectively.
The Riemann-Hurwitz formula implies that
$\deg\mathfrak{d}_u=2\deg u-2$ and
$\deg\mathfrak{d}_v=2\deg v-2$.
Since all finite poles of $w$ are simple,
the same holds for the finite poles of $u$ and $v$.
Moreover, since $w$ is tamely ramified, so are $u$ and $v$,
which therefore each ramify over at least two geometric primes.
Thus, $\mathfrak{d}_u^0> 0$ and $\mathfrak{d}_v^0>0$.
By the basic properties of differents we have 
$\mathfrak d_w=u^*\mathfrak d_v+\mathfrak d_u$.
Therefore, as all divisors in this relation are effective,
\begin{equation}\label{eq:indec1}
 \mathfrak d_w^0=u^*\mathfrak d_v^0+\mathfrak d_u^0.
\end{equation}
Note that $\mathfrak d_w^0$ is precisely the zero divisor of the numerator of $w'$,
which by assumption is of the form
$\mathfrak d_w^0=Q$ or $\mathfrak d_w^0=2Q$ for some prime divisor $Q$.
The first case is already excluded by (\ref{eq:indec1}) (since $u^*\mathfrak d_v^0,\mathfrak d_u^0>0$, their sum cannot be prime), 
so assume that $\mathfrak d_w^0=2Q$. 
Then (\ref{eq:indec1}) gives that
$u^*\mathfrak d_v^0=\mathfrak d_u^0=Q$.
We have
$$
 (\deg u)(\deg\mathfrak d_v^0)=\deg u^*\mathfrak{d}_v^0=\deg Q=\deg \mathfrak d_u^0\le\deg\mathfrak d_u=2\deg u-2,
$$
which implies $(\deg u)(\deg\mathfrak{d}_v^0-2)\leq -2$,
hence $\deg u=2$ and $\deg\mathfrak d_v^0=1$. 
Since $\deg\mathfrak{d}_v^0+\deg\mathfrak{d}_v^\infty=2\deg v-2$, the latter shows that $\mathfrak d_v^\infty>0$.
Since $\deg\mathfrak{d}_u^0+\deg\mathfrak{d}_u^\infty=2\deg u-2=2$
and $\deg\mathfrak{d}_u^0=\deg Q=(\deg u)(\deg\mathfrak{d}_v^0)\geq 2$ it also follows that $\mathfrak{d}_u^\infty=0$.
Thus $v$ ramifies over $\infty$ but $u$ does not,
which implies that $w=v\circ u$ has $\deg u=2$ many
multiple poles over $\overline{k}$, contradicting our assumption.
\end{proof}

\begin{proposition} \label{prop:monfq}
Let $p$ be prime, $q$ a power of $p$, and $m\ge 0$ and $n\ge m+2$ integers with $n\not\equiv m\pmod p$. 
Let $f,c\in\F_q[X]$ be polynomials with ${\rm deg}(f)=n$, ${\rm deg}(c)=m$, $(f,c)=1$ and $c$ squarefree.
Let $w=\frac{f}{c}$.
\begin{enumerate}
\item[(i)] Assume $p>2$, $m\equiv n\pmod 2$, $m\neq 2$ and $g=f'c-fc'$ is irreducible. Then $\Mon_{\Fb_q}(w)=S_n$ or 
\begin{itemize}
\item $n=9,m=1$, $\Mon_{\Fb_q}(w)=AGL_2(3)$ with ram.~type $(1^32^3,1^32^3,1^32^3,1^18^1)$
\end{itemize}

\item[(ii)] Assume $p>3$, $m\not\equiv n\pmod 2$, and $f'c-fc'=g^2$ where $g\in \mathbb{F}_q[X]$ is irreducible. Then $\Mon_{\Fb_q}(w)=A_n$ or one of the following:
\begin{itemize}
\item $n=8,m=1,\Mon_{\Fb_q}(w)=A\Gamma L_1(8)$ or $L_2(7)$ and ram.~type $(1^23^2,1^23^2,1^17^1)$.
\item $n=9,m=0,\Mon_{\Fb_q}(w)=P\Gamma L_2(8)$ and ram.~type $(1^33^2,1^33^2,9^1)$.
\item $n=12,m=1,\Mon_{\Fb_q}(w)=M_{12}$ and ram.~type $(1^33^3,1^33^3,1^111^1)$.
\item $n=24,m=1,\Mon_{\Fb_q}(w)=M_{24}$ and ram.~type $(1^63^6,1^63^6,1^123^1)$.
\end{itemize}
\end{enumerate}
\end{proposition}

\begin{proof} 
Note that $\Mon_{\Fb_q}(w)\unlhd\Mon_{\F_q}(w)$ are transitive
subgroups of $S_n$.
By Lemma~\ref{lem:indec}, $w$ is indecomposable over $\F_q$,
hence $\Mon_{\F_q}(w)$ is primitive (Lemma \ref{lem:indecomposable_primitive}),
and Lemma~\ref{lemcycle}(ii) applied to the infinite prime
shows that $\Mon_{\Fb_q}(w)$ contains an $(n-m)$-cycle.
We treat parts $(i)$ and $(ii)$ simultaneously but distinguish several cases according to the value of $m$:

Case $m=0$: In this case $w$ 
is a polynomial. By \cite[Theorem 3.5]{FrMa69} a polynomial which is indecomposable over $\F_q$ of degree coprime with $p$ is also indecomposable over $\Fb_q$ and therefore in our case $w$ is indecomposable over $\Fb_q$. 
The conditions of Proposition~\ref{prop:mon}(i) resp.~(ii) hold and therefore 
$\Mon_{\Fb_q}(w)=S_n$ (none of the exceptional cases has $m=0$) resp.~$\Mon_{\Fb_q}(w)=A_n$ or the exceptional case with $m=0$ appearing in Proposition~\ref{prop:mon}(ii).

Case $m=1$:
$\Mon_{\Fb_q}(w)$ is transitive and contains an $(n-1)$-cycle and is therefore primitive, hence $w$ is indecomposable over $\Fb_q$ (Lemma \ref{lem:indecomposable_primitive}). The conditions of Proposition~\ref{prop:mon}(i) resp.~(ii) hold and therefore $\Mon_{\Fb_q}(w)=S_n$ or the exceptional case with $m=1$ in Proposition~\ref{prop:mon}(i) resp.~$\Mon_{\Fb_q}(w)=A_n$ or one of the exceptional cases with $m=1$ listed in Proposition~\ref{prop:mon}(ii).

Case $m=2$: Note that this case is allowed only in part $(ii)$. 
Since $n\ge 5$ is odd by assumption, $(n-2)$ is coprime with $n$ and so $\Mon_{\Fb_q}(w)$ must be primitive since it is transitive and contains an $(n-2)$-cycle. Hence $w$ is indecomposable over $\Fb_q$ (Lemma~\ref{lem:indecomposable_primitive}). The conditions of Proposition~\ref{prop:mon}(ii) hold and therefore $\Mon_{\Fb_q}(w)=A_n$.

Case $m>2$: Here $\Mon_{\F_q}(w)$ is primitive and contains an $(n-m)$-cycle, so by Theorem~\ref{thm:jones}(i) it is $S_n$ or $A_n$.
Note that $n\geq m+2\geq 5$, so $A_n$ is simple and $S_n$ has a unique nontrivial cyclic quotient.
Now $\Mon_{\Fb_q}(w)\unlhd\Mon_{\F_q}(w)$ and $\Mon_{\F_q}(w)/\Mon_{\Fb_q}(w)$ is cyclic,
hence also $\Mon_{\Fb_q}(w)$ is $S_n$ or $A_n$.
Under the assumptions of part $(i)$, the $(n-m)$-cycle is odd, and thus $\Mon_{\Fb_q}(w)=S_n$.
Under the assumptions of part $(ii)$, the $(n-m)$-cycle is even, as are all the other inertia subgroups
(the ramification indices at the finite ramified primes all equal $3$),
so since $\Mon_{\Fb_q}(w)$ is generated by the inertia subgroups,
$\Mon_{\Fb_q}(w)=A_n$.
\end{proof}

For a squarefree polynomial $c\in\F_q[X]$ we denote 
$$
 \mathcal{H}_c=\left\{R^p: R\in\left(\F_q[X]/c^2\F_q[X]\right)^\times\right\}\leqslant\left(\F_q[X]/c^2\F_q[X]\right)^\times.
$$
Denote also 
\begin{equation}\label{eq:def_pi_c}
 \pi_c=\prod_{i=1}^m(X-\alpha_i)^2\cdot\sum_{i=1}^m\frac 1{(X-\alpha_i)^2}=\sum_{i=1}^m\prod_{j\neq i}(X-\alpha_j)^2\in\F_q[X],
\end{equation}
where $\alpha_1,\dots,\alpha_m\in\Fb_q$ are the roots of $c$. Note that $(\pi_c,c)=1$.
In short we write $\pi_c\mathcal{H}_c$ for the coset $(\pi_c+c^2\F_q[X])\mathcal{H}_c$ in $(\F_q[X]/c^2)^\times$.

\begin{lemma} \label{lem:der} 
Let $c\in\F_q[X],\deg c=m$ be a squarefree polynomial and $g\in\F_q[X]$ a polynomial such that $(g,c)=1$ and $\deg g\ge 2m$. 
Denote $n=\deg(g)-m+1$ and assume that $p=\mathrm{char}(\F_q)>n-m$. 
Assume further that $g\bmod c^2\in\pi_c\mathcal H_c$. 
Then there exists $f\in\F_q[X],\deg f=n,(f,c)=1$ such that $g=f'c-fc'$.
\end{lemma}

\begin{proof} 
Without loss of generality, $c$ is monic.
Write $g=uc^2+v,u,v\in\F_q[X],\deg v<\deg c^2=2m$. Let $\alpha_1,\ldots,\alpha_m\in\Fb_q$ be the 
roots of $c$ (they are pairwise distinct since $c$ is squarefree). By assumption $v\equiv\pi_cw\pmod{c^2}
$, where $w$ is a $p$-th power modulo each $(X-\alpha_i)^2$, the latter condition implying $w\equiv 
a_i\pmod {(X-\alpha_i)^2}$ for a (uniquely determined) $0\neq a_i\in\Fb_q$. 
It follows using (\ref{eq:def_pi_c}) that $v\equiv a_i\prod_{j\neq i}(X-\alpha_j)^2\pmod{(X-\alpha_i)^2}$ for each $1\le i\le m$,
hence $v\equiv \sum_{i=1}^ma_i\prod_{j\neq i}(X-\alpha_j)^2\pmod{c^2}$. 
Since both sides of this congruence are of degree less than $\deg c^2$,
we get  the following
partial fraction decomposition over $\Fb_q$:
$$
 \frac v{c^2}=\sum_{i=1}^m\frac{a_i}{(X-\alpha_i)^2}.
$$ 
We therefore can write
$$
 -\sum_{i=1}^m\frac{a_i}{X-\alpha_i}=\frac\psi c\quad\mbox{ with }\quad\psi\in\Fb_q[X],\deg\psi<m,
$$ 
and 
since the map $\{\alpha_1,\dots,\alpha_m\}\rightarrow\{a_1,\dots,a_m\}$, $\alpha_i\mapsto a_i$,
is a morphism of ${\rm Gal}(\Fb_q/\F_q)$-sets,
in fact $\psi\in\F_q[X]$.
We have $(\psi/c)'=v/c^2$.
Now $\deg u=n-m-1<p-1$ and therefore there exists $\phi\in\F_q[X],\deg\phi=n-m$ such that $\phi'=u$. 
Denote $f=\phi c+\psi$. We have $\deg f=n$ and 
$$\left(\frac fc\right)'=\phi'+\left(\frac \psi c\right)'=
u+\frac v{c^2}=\frac g{c^2},$$
from which it follows that $g=f'c-fc'$.
\end{proof}

\begin{proposition}\label{prop:g} Let $p$ be prime, $q$ a power of $p$, $m\ge 0,n\ge m+2$ integers such that $p>n-m$. 
Let $c\in\F_q[X],\deg c=m$ be squarefree.
\begin{enumerate}
\item[(i)] Assume that $q^{\frac {n-m-1}2}>2m+1$. 
Then there exists $f\in\F_q[X]$, $\deg f=n$, $(f,c)=1$, such that $g=f'c-fc'$ is irreducible in $\mathbb{F}_q[X]$.
\item[(ii)] Assume $n\not\equiv m\pmod 2$ and $q^{\frac {n-3m-1}4}>2m+1$. Then there exists $f\in\F_q[X]$, $\deg f=n$, $(f,c)=1$, such that $f'c-fc'=g^2$ with $g\in\mathbb{F}_q[X]$ irreducible.
\end{enumerate}
\end{proposition}

\begin{proof} 
In case (i) let $N=n+m-1$ and $\psi_c=\pi_c$;
in case (ii) let $N=\frac{n+m-1}{2}$ and choose $\psi_c\in\F_q[X]$ with  $\psi_c^2\equiv\pi_c\pmod{c^2}$.
The latter exists since $\pi_c$ is a square modulo $c$ (e.g.~$\pi_c\equiv\delta^2\pmod c$ with
$\delta=\prod_{i=1}^m(X-\alpha_i)\cdot\sum_{i=1}^m\frac 1{(X-\alpha_i)}$)
and the kernel of $(\F_q[X]/c^2)^\times\rightarrow(\F_q[X]/c)^\times$
has order $q^m$, which is odd.

By Lemma~\ref{lem:der} it suffices to show that there is a monic irreducible polynomial $g\in\F_q[X]
$ with $\deg g=N$ such that $g\bmod c^2\in\mathcal \psi_c\mathcal{H}_c$ (for case $(ii)$ note that $g\bmod c^2\in
\psi_c\mathcal{H}_c$ implies $g^2\bmod c^2\in\pi_c\mathcal{H}_c$).
Denote by $\mathcal H_c^\perp$ the group of Dirichlet characters modulo $c^2$ that are trivial on $\mathcal H_c$ (this is the orthogonal group of $\mathcal H_c$). 
We have $|\mathcal H_c^\perp|=[(\F_q[X]/c^2)^\times:\mathcal H_c]=q^m$. 
Below when we sum over $g$ we will always restrict $g$ to be monic,
and $\Lambda$  denotes the polynomial von Mangoldt function, 
i.e.~$\Lambda(g)=\deg P$ if $g=P^\nu$ for a monic irreducible $P\in\mathbb{F}_q[X]$ and $\nu\geq1$, and $\Lambda(g)=0$ otherwise. 
By the second orthogonality relation \cite[Corollary on p. 63]{Ser73} applied to the group $\left(\F_q[X]/c^2\right)^\times/\mathcal H_c$ (the dual of this group can be identified with $\mathcal H_c^\perp$) and the element $g/\psi_c$, we have

$$\sum_{\chi\in\mathcal H_c^\perp}\overline{\chi(\psi_c)}\chi(g)=\left[\begin{array}{ll}q^m,&\mbox{if } g\bmod c^2\in\psi_c\mathcal H_c,\\
0,&{\mathrm {otherwise}},\end{array}\right.$$ and hence

$$
 \sum_{\deg g=N\atop{g \bmod c^2\in\psi_c\mathcal H_c}}\Lambda(g)=
\frac 1{q^m}\sum_{\chi\in\mathcal H_c^\perp}\overline{\chi(\psi_c)}\sum_{\deg g=N}\chi(g)\Lambda(g).
$$
Let $\chi_1$ denote the trivial character modulo $c^2$.
For all $\chi\neq\chi_1$,
Weil's Riemann Hypothesis for function fields gives that\footnote{For example using the notation of \cite[p.~41-42]{Ros02}:
$\sum_{\deg g=N}\chi(g)\Lambda(g)=c_N(\chi)=-\sum_{k=1}^{M-1}\alpha_k(\chi)^N$
where $M=\deg c^2$ and $|\alpha_k(\chi)|\leq\sqrt{q}$.}
$$
 \left|\sum_{\deg g=N}\chi(g)\Lambda(g)\right|\le (2m-1)q^{\frac{N}2}.
$$
Thus together with the
elementary relation $\sum_{\deg g=N}\Lambda(g)=q^N$ (see \cite[Proposition 2.1]{Ros02}) we obtain
\begin{multline*}
\sum_{\deg g=N\atop{g \bmod c^2\in\psi_c\mathcal H_c}}\Lambda(g)\ge
\frac{1}{q^m}\sum_{\deg g=N\atop{(g,c)=1}}\Lambda(g)-\frac{q^m-1}{q^m}(2m-1)q^{\frac{N}2}\ge 
 q^{N-m}-1-(2m-1)q^{\frac{N}2}.
\end{multline*}
By the same elementary relation,
$$
 \sum_{\deg g=N\atop{g\bmod c^2\in\mathcal \psi_cH_c\atop{g\,\mathrm{not\,irreducible}}}}\Lambda(g)\le\sum_{d|N\atop{d\neq N}}q^d\le 2q^{\frac{N}2}-1,
$$ 
 and therefore as long as
$$
 q^{N-m}>(2m+1)q^{\frac{N}2}
$$ 
there must be at least one irreducible $g$ with $\deg g=N$ and $g\bmod c^2\in\mathcal \psi_c\mathcal H_c$.
To conclude the proof note that $\frac{N}{2}-m$ equals $\frac{n-m-1}{2}$ in case $(i)$
and $\frac{n-3m-1}{4}$ in case $(ii)$.
\end{proof}

{
\allowdisplaybreaks

\begin{theorem}\label{thm:tame2} 
Let $p$ be a prime, $q$ a power of $p$, $0\le m\le n-2$ integers such that $n-m<p$.
Assume that either 
\begin{enumerate}
\item[(i)] $G=S_n$ and the following conditions hold
\begin{enumerate}
\item $p>2$,
\item $m\equiv n\pmod 2$ and $m\neq 2$, 
\item $(n,m)\neq (9,1)$, and
\item $q^{\frac{n-m-1}2}>2m+1$,
\end{enumerate}
or
\item[(ii)] $G=A_n$ and the following conditions hold
\begin{enumerate}
\item $p>3$,
\item $m\not\equiv n\pmod 2$ and $(n,p)=1$,
\item $(n,m)\neq (8,1),(9,0),(12,1),(24,1)$,
\item $q^{\frac{n-3m-1}4}>2m+1$, and
\item $m\ge 2$ or $\left(\frac q{n-m}\right)=1$.
\end{enumerate}
\end{enumerate}
Then there exist $f,c\in\mathbb{F}_q[X]$ with ${\rm deg}(f)=n$, ${\rm deg}(c)=m$, and $c$ squarefree
such that the splitting field of $f(X)-Tc(X)$ over $\mathbb{F}_q(T)$ is geometric with Galois group $G$
and is ramified
over only one finite prime $\mathcal F$,
with $\Disc_X(f-Tc)$ a power of $\mathcal F$ and ${\rm deg}(\Disc_X(f-Tc))=n+m-1$.
In particular, $\ram_{\mathbb{F}_q(T)}(G)\leq 2$.
\end{theorem}

\begin{proof} 

$(i)$. First assume that $G=S_n$, $p>2$, $m\neq 2$, $n\equiv m\pmod 2$, $(n,m)\neq (9,1)$ and $q^{\frac{n-m-1}2}>2m+1.$ 
Take any squarefree $c\in\F_q[X],\deg c=m$. 
By Proposition~\ref{prop:g}(i) one can find an $f\in\F_q[X],\deg f=n$ such that $g=f'c-fc'$ is irreducible of degree $n+m-1$. By Proposition~\ref{prop:monfq}(i) the splitting field $K$ of $f(X)-Uc(X)$ over $\F_q(U)$ is geometric with Galois group $S_n$. 
Let $w=\frac{f}{c}\in\F_q(X)$.

Consider the discriminant $D(U)=\Disc_X(f(X)-Uc(X))$. Let 
$$
 \alpha_1,\alpha_2=\alpha_1^q,\ldots,\alpha_{n+m-1}=\alpha_1^{q^{n+m-2}}\in\F_{q^{n+m-1}}
$$ 
be the roots of $g$. 
Then by Lemma~\ref{lem:crit} the roots of $D$ are ${w(\alpha_i)},1\le i\le n+m-1$ (including multiplicity). The 
Frobenius map $\Fr_q$ acts cyclically on ${w(\alpha_i)}$ and hence $D=\mathcal{F}^r$ is a 
power of some prime $\mathcal{F}\in\F_q[U]$. 
By Lemma~\ref{lem:crit}, $w$ is ramified only over $\mathcal F$ and (possibly) $\infty$, which concludes the proof in the case $G=S_n$.

$(ii)$.
Now assume that $G=A_n$, $p>3$, $n\not\equiv m\pmod 2$, $(n,m)\neq (8,1)$, $(9,0)$, $(12,1)$, $(24,1)$ and $q^{\frac{n-3m-1}4}>2m+1$.
Let $c\in\F_q[X],\deg c=m$ be a monic squarefree polynomial. 
By Proposition~\ref{prop:g}(ii) there exists a monic irreducible $g\in\F_q[X]$ and a monic $f\in\F_q[X],\deg f=n$ such that $f'c-fc'=(n-m)g^2$ (we can always adjust the leading coefficients this way). 
Note that $\deg f'=n-1$ since $(n,p)=1$.
Let $w=\frac{f}{c}\in\F_q(X)$ and let $K$ be the splitting field of $f(X)-Uc(X)$ over $\F_q(U)$.
By Proposition~\ref{prop:monfq}(ii),
$\Mon_{\Fb_q}(w)=A_n$, hence $\Gal(K/\F_q(U))=\Mon_{\F_q}(w)$ is either $A_n$ or $S_n$ and $K/\F_q(U)$ is geometric if it is $A_n$, which happens precisely when $D(U)=\Disc_X(f(X)-Uc(X))$ is a square in $\F_q(U)$ (Lemma~\ref{lem:discgalois}).

To shorten notation, we  omit the variable $X$ in the following calculation. All discriminants and resultants are with respect to the variable $X$. Terms which are squares in $\F_q(U)$ are of no consequence and are simply denoted by $\square$.
If $\deg c\ge 1$ then using the properties of resultants and discriminants 
we calculate:
\begin{eqnarray*}
D(U)&\stackrel{(\ref{eq:disc})}=&(-1)^{\frac{n(n-1)}2}\Res(f'-Uc',f-Uc)\\
&\stackrel{(\ref{eq:resbimult})}=&(-1)^{\frac{n(n-1)}2}\Res(c,f-Uc)^{-1}\Res(f'c-Uc'c,f-Uc)\\
&\stackrel{(\ref{eq:resalt})}=&\square\cdot(-1)^{\frac{n(n-1)}2}\Res(c,f)\Res(f'c-fc'+c'(f-Uc),f-Uc)\\
&\stackrel{(\ref{eq:resalt}),(\ref{eq:ressym})}=&\square\cdot(-1)^{\frac{n(n-1)}2}\Res(c,f)\Res(f'c-fc',f-Uc)\\
&\stackrel{(\ref{eq:resbimult}),(\ref{eq:resdefinition})}=&\square\cdot(-1)^{\frac{n(n-1)}2}\Res(c,c')\Res(c,c'f)(n-m)^n\prod_{g(\alpha)=0}(f(\alpha)-Uc(\alpha))^2\\
&\stackrel{(\ref{eq:resalt})}=&\square\cdot(-1)^{\frac{n(n-1)}2}\Res(c,c')\Res(c,c'f-cf')(n-m)^n\\
&\stackrel{(\ref{eq:resbimult})}=&\square\cdot(-1)^{\frac{n(n-1)}2}\Res(c,c')\Res(c,m-n)\Res(c,g)^2(n-m)^n\\
&\stackrel{(\ref{eq:resdefinition})}=&\square\cdot(-1)^{\frac{n(n-1)}2+m}\Res(c,c')(n-m)^{n+m}\\
&\stackrel{(\ref{eq:disc}),(\ref{eq:ressym})}=&\square\cdot(-1)^{\frac{n(n-1)}2+\frac{m(m+1)}2}(n-m)\Disc(c).
\end{eqnarray*}
If $c=1$ a similar calculation gives the same final expression for $D(U)$.

Therefore, $D(U)$ is a square in $\F_q(U)$ if and only if 
$$
 \delta:=(-1)^{\frac{n(n-1)}2+\frac{m(m+1)}2}(n-m)\Disc(c)
$$ 
is a square in $\mathbb{F}_q$.
If $m<2$, then $\Disc(c)=1$, and
by quadratic reciprocity,
$$
 \left(\frac{(-1)^{\frac{n(n-1)}{2}+m}(n-m)}{p}\right)
 =(-1)^{\frac{p-1}{4}(n^2+m-1)}\left(\frac{p}{n-m}\right)
 =\left(\frac{p}{n-m}\right)
$$
as $n^2+m-1\equiv 0\pmod 4$ if either $m=0$ and $n\equiv1\pmod 2$, or $m=1$ and $n\equiv0\pmod 2$,
thus $\delta$ is a square in $\mathbb{F}_q$ if and only if $\left(\frac q{n-m}\right)=1$, which is precisely our assumption 
(e) for this case.
If $m\geq2$, then
$\Disc(c)$ is a square in $\mathbb{F}_q$ if and only if
$(-1)^m\mu(c)=1$ where $\mu(c)$ 
is the polynomial M\"obius function, see \cite[Lemma 4.1]{Conrad}.
Therefore we can choose $c$ accordingly with an odd or an even number of irreducible factors
to get the value we need so that $\delta$ is a square in $\F_q$. 

This shows that $K/\F_q(U)$ is geometric with Galois group $A_n$.
Arguing as in $(i)$ one sees that $D=\mathcal F^r$ with $\mathcal F$ prime and that $w$ is ramified only over $\mathcal F$ and possibly $\infty$.
\end{proof}

}

\subsection{Eliminating tame ramification at infinity}
\label{sec:oneramified}

We will now explain how to eliminate
the tame ramification at infinity of the extensions obtained in the previous subsections,
conditional on a weak consequence of the function field analogue of Schinzel's hypothesis H
that can actually be proven in many cases.
The classical Hypothesis H states the following:

\begin{conjecture}[Schinzel's hypothesis H]\label{conj:Schinzel}
Let $f_1,\dots,f_r\in\mathbb{Z}[X]$ irreducible.
If $f_1\cdots f_r$ is not the zero function modulo any prime number $p$,
then there exist infinitely many $n\in\mathbb{Z}$ with $f_1(n),\dots,f_r(n)$ simultaneously prime.
\end{conjecture}

In the function field setting, 
the naive analogue fails due to inseparability issues,  
see \cite{CCG}.
Taking this into account,
one conjectures:

\begin{conjecture}\label{conj:SchinzelFF}
Let $\mathcal{F}_1,\dots,\mathcal{F}_r\in\mathbb{F}_q[T][X]$ irreducible and separable in $X$.
If $\mathcal{F}_1\cdots \mathcal{F}_r$ is not the zero function modulo any prime polynomial $P\in\mathbb{F}_q[T]$,
then there exist infinitely many $h\in\mathbb{F}_q[T]$ with $\mathcal{F}_1(h),\dots,\mathcal{F}_r(h)$ simultaneously irreducible in $\mathbb{F}_q[T]$.
\end{conjecture}

For partial results on Conjecture \ref{conj:SchinzelFF} see for example
\cite{Pollack,BS12,EntinBH,Entin,SawinShusterman}.
We will work with the following weaker assumption:

\begin{definition} 
We say that property $\HH(q,d,e)$ holds if for any irreducible $\mathcal{F}\in\F_q[X]$ with $\deg(\mathcal{F})|d$ there exists $h\in\F_q[T]$ with $e|\deg(h)$ such that 
$\mathcal{F}(h)$ is irreducible in $\mathbb{F}_q[T]$. 
\end{definition}

We write ${\rm rad}(e)$ for the radical of $e$
and $\omega(e)$ for the number of distinct prime divisors of $e$.
We define ${\rm rad}'(e)=\frac{\gcd(4,e)}{\gcd(2,e)}{\rm rad}(e)$,
i.e.~${\rm rad}'(e)$ equals $2{\rm rad}(e)$ if $4|e$ 
and ${\rm rad}(e)$ otherwise.
We denote by $v_\ell$ the $\ell$-adic valuation.

\begin{proposition}\label{lem:Hqdm}
$\HH(q,d,e)$ holds in each of the following cases:
\begin{enumerate}
\item Conjecture \ref{conj:SchinzelFF} holds for $\mathbb{F}_q$.
\item $q$ is sufficiently large with respect to $d$ and $e$.
\item ${\rm rad}'(e)|q-1$ and $\gcd(d,e)=1$.
\item ${\rm rad}'(e)|q^\ell-1$ for every prime divisor $\ell$ of $d$, and $q\geq(d-1)^2(2^{\omega(e)}-1)^2$.
\item ${\rm rad}(e)|q-1$ and $q\geq(2^{\max\{v_2(e)-1,0\}}d-1)^2(2^{\omega(e)}-1)^2$.
\end{enumerate}
More precisely, 
for any irreducible $\mathcal{F}\in\F_q[X]$ with $\deg(\mathcal{F})|d$ there exists $h\in\F_q[T]$ with $e|\deg(h)$ such that 
$\mathcal{F}(h)$ is irreducible in $\mathbb{F}_q[T]$,
where
\begin{enumerate}[(a)]
\item there exists such $h$ with ${\rm deg}(h)=e$ in cases (2), (3), (4) and (5),
\item there exists such $h$ that is monic in cases (1), (2), (4) and (5), and
\item there exists such $h$ that is monic and satisfies ${\rm deg}(h)=e$ in cases (4) and (5), as well as in case (2) if $q$ is odd or $e\geq 4$.
\end{enumerate}
\end{proposition}

\begin{proof}
Let $e,d\geq 1$ and $\mathcal{F}\in\mathbb{F}_q[X]$ irreducible with ${\rm deg}(\mathcal{F})=n|d$ be given.

(1): 
Without loss of generality assume that $q-1|e$.
The polynomial
$$
 g:=\mathcal{F}(X^e+A_{e-1}X^{e-1}+\dots+A_0)\in\mathbb{F}_q[A_0,\dots,A_{e-1},X]
$$ 
is separable in $X$ and irreducible:
Indeed, $g$ is obtained from $\mathcal{F}(A_0)$ by the invertible change of variables $A_0\mapsto X^e+A_{e-1}X^{e-1}+\dots+A_0$.
Let $S$ be the set of primes $P\in\mathbb{F}_q[T]$ with ${\rm deg}(P)\leq\log_q{\rm deg}_Xg$
and let $M=\prod_{P\in S}P$.
By \cite[Theorem 1.1]{BaEn19_} there exist $a_0\in\mathbb{F}_q+M\mathbb{F}_q[T]$ and $a_1,\dots,a_{e-1}\in\mathbb{F}_q[T]$ with
$\tilde{g}(X):=g(a_0,\dots,a_{e-1},X)\in\mathbb{F}_q[T][X]$ separable in $X$ and irreducible in $\mathbb{F}_q(T)[X]$,
hence also irreducible in $\mathbb{F}_q[T,X]$ (since it is primitive in $X$).
In particular, $\tilde{g}(0)=\mathcal{F}(a_0)\not\equiv 0\mbox{ mod }P$ for every $P\in S$,
and modulo a prime $P\in\mathbb{F}_q[T]$ not in $S$, 
$\tilde{g}$ has at most ${\rm deg}_X\tilde{g}<q^{{\rm deg}(P)}$ many zeros.
Therefore, Conjecture \ref{conj:SchinzelFF} predicts
infinitely many $h\in\mathbb{F}_q[T]$ such that,
with $\tilde{h}:=h^e+a_{e-1}h^{e-1}+\dots+a_0$,
the polynomial
$\tilde{g}(h)=\mathcal{F}(\tilde{h})$ is irreducible.
For ${\rm deg}(h)$ sufficiently large, ${\rm deg}(\tilde{h})={\rm deg}(h^e)$ is divisible by $e$,
and ${\rm lc}(\tilde{h})={\rm lc}(h^e)$,
so $\tilde{h}$ is monic as $q-1|e$.

(2): For large $q$, $H(q,d,e)$ follows immediately from \cite[Theorem 1.1]{EntinBH}.
That one can choose $h$ monic of degree $e$ 
follows in the case that $q$ is large and odd from
\cite[Theorem 1.4]{BS12}.
If $q$ is large and $e\geq 4$, the existence of a suitable monic $h$ of degree $e$
follows from
\cite[Proposition 6.1]{Entin}:
If $\mathcal{F}(X)$ is irreducible, then so is $\mathcal{F}(T^e+X)$,
and \cite[Proposition 6.1]{Entin} gives a $h_0$ of degree $e-1\geq 3$
with $\mathcal{F}(T^e+h_0)$ irreducible.
This also implies 
that if $q$ is large one can always obtain a suitable $h$ that is monic,
by replacing $e$ by $4e$ if necessary.

(3): 
Let $\alpha\in\mathbb{F}_{q^n}$ be a root of $\mathcal{F}$.
We will show that there exists $c\in\mathbb{F}_q^\times$ such that $c\alpha$ is not an $\ell$-th power in $\mathbb{F}_{q^n}^\times$
for any prime divisor $\ell$ of $e$.
Then $T^e-c\alpha\in\mathbb{F}_{q^n}[T]$ is irreducible, cf.~\cite[Chapter VI, Theorem 9.1]{Lang},
where in the case $4|e$ one has to observe 
that the assumption $2{\rm rad}(e)|q-1$ implies that $-4$ is a square in $\mathbb{F}_q^\times$, hence
$-4c\alpha$ is in particular not a $4$-th power in $\mathbb{F}_{q^n}^\times$.
This will show that $\mathcal{F}(h)$ is irreducible for $h=c^{-1}T^e$.

Let $y=N_{\mathbb{F}_{q^n}/\mathbb{F}_q}(\alpha)$.
By the assumptions that ${\rm rad}(e)|q-1$
and that $n$ and $e$ are coprime
we can apply the isomorphism
$\mathbb{F}_q^\times\cong\mathbb{Z}/(q-1)\mathbb{Z}\cong\prod_{\ell|q-1}\mathbb{Z}/\ell^{v_\ell(q-1)}\mathbb{Z}$
to obtain $c$ such that $c^ny$ is not an $\ell$-th power in $\mathbb{F}_q^\times$ for any prime divisor $\ell$ of $e$.
As the norm is multiplicative and  $N_{\mathbb{F}_{q^n}/\mathbb{F}_q}(c\alpha)=c^ny$, the element $c\alpha$ is then not an $\ell$-th power in $\mathbb{F}_{q^n}^\times$ as well.

(4): The case ${\rm deg}(\mathcal{F})=1$ is trivial, so assume that $n>1$
and note that the assumption implies that ${\rm rad}'(e)|q^n-1$.
Let $\alpha\in\mathbb{F}_{q^n}$ be a root of $\mathcal{F}$.
Cohen's result \cite[Corollary 2.3]{Cohen} gives that
the number of $c\in\mathbb{F}_q$ for which $\alpha+c$ is not an $\ell$-th power in $\mathbb{F}_{q^n}$ for any prime divisor $\ell$ of $e$ is bigger than
$\frac{\phi({\rm rad}(e))}{{\rm rad}(e)}(q-(n-1)(2^{\omega(e)}-1)\sqrt{q})$,
in particular this number is positive as soon as
$q\geq(n-1)^2(2^{\omega(e)}-1)^2$, which is satisfied by our assumption.
As in (3) we conclude that $T^e-(\alpha+c)\in\mathbb{F}_{q^n}[T]$ is irreducible,
hence $\mathcal{F}(h)$ is irreducible for $h=T^e-c$.

(5): Let $s=v_2(e)$. If $s\leq 1$, then ${\rm rad}'(e)={\rm rad}(e)$, so the claim follows from (4).
If $s\geq 2$, then $q$ is odd and we first apply (4) with $e=2$ to obtain a monic $h_1\in\mathbb{F}_q[T]$ of degree $2$ with $\mathcal{F}_1(T):=\mathcal{F}(h_1(T))$ irreducible of degree $2n$. 
We iterate this until we obtain $\mathcal{F}_{s-1}=\mathcal{F}(h_{1}(\dots h_{s-1}(T)))$
irreducible of degree $2^{s-1}n$.
Then ${\rm rad}'(e/2^{s-1})={\rm rad}(e/2^{s-1})$,
so we can apply (4)
to obtain a monic $h_s$ of degree $e/2^{s-1}$ with $\mathcal{F}(h_1(\dots h_s(T)))$
irreducible.
The polynomial $h=h_1\circ\dots\circ h_s$ has degree $e$ and satisfies the claim.
\end{proof}

\begin{lemma}\label{lem:disjoint} 
Let $k$ be a perfect field and let $h\in k[T]$ be a non-constant polynomial. 
Denote $U=h(T)$ and let $K$ be a Galois extension of $k(U)$ ramified only over $U=\infty$ and the finite primes $\mathcal{F}_1,\dots,\mathcal{F}_r\in k[U]$, with tame ramification over $\infty$. 
Assume that each $\mathcal{F}_i(h(T))\in k[T]$ is separable. 
Then the extensions $K,k(T)$ of $k(U)$ are linearly disjoint.
\end{lemma}

\begin{proof} 
Denote $D(U)=\Disc_T(h(T)-U)$. 
By Lemma~\ref{lem:crit} the finite branching primes of the extension $k(T)/k(U)$ are exactly the primes dividing $D(U)$.
If $\mathcal{F}_i|D$ for some $i$, then 0 would be a critical value of $\mathcal{F}_i\circ h\colon\P^1\to\P^1$, which means that $\mathcal{F}_i(h)$ is not separable, contradicting our assumption. Thus $\mathcal F_i\nmid D$ for each $i$,
so the finite branching loci of $k(T)/k(U)$ and $K/k(U)$ are disjoint,
hence the extension $K\cap k(T)/k(U)$ is tamely ramified and ramified only over $\infty$, and is therefore trivial.
Since $K/k(U)$ is Galois, this already implies that $K$ and $k(T)$ are linearly disjoint over $k(U)$. 
\end{proof}

\begin{lemma}\label{lem:eliminateinfty}
Let $k$ be a perfect field of characteristic $p\geq 0$, and let $f\in k[U,X]$.
Assume that the splitting field of $f$ over $k(U)$
is ramified only over the primes $\mathcal{F}_1,\dots,\mathcal{F}_r\in k[U]$ 
and over the infinite prime with ramification index $e$ 
not divisible by $p$.
If $h\in k[T]$ is such that $e|{\rm deg}(h)$,
then the splitting field of $f(h(T),X)$ over $k(T)$
is ramified at most over the prime factors of $\mathcal{F}_1(h(T)),\dots,\mathcal{F}_r(h(T))$.
\end{lemma}

\begin{proof}
We identify $U$ with $h(T)$ and $k(U)=k(h(T))\subseteq k(T)$. 
Let $K$ be the splitting field of $f(U,X)$ over $k(U)$,
and consider the splitting field $L$ of $f(h(T),X)$ over $k(T)$, which is the compositum of $K$ and $k(T)$. 
Since $k(T)/k(U)$ is totally ramified over $\infty$ of degree ${\rm deg}(h)$ divisible by $e$,
by Abhyankar's Lemma (Lemma~\ref{lem:abhyankar}) the base change from $k(U)$ to $k(T)$ eliminates the ramification over infinity. 
Every finite prime of $k(T)$ that ramifies in $L$ lies over a 
finite prime of $k(U)$ that ramifies in $K$,
i.e.~is a prime factor of some $\mathcal{F}_i(h)$.
\end{proof}

\begin{theorem}\label{thm:tame} 
Let $p$ be a prime, $q$ a power of $p$, $0\le m\le n-2$ integers such that $n-m<p$.
Assume that either 
\begin{enumerate}
\item[(i)] $G=S_n$ and the following conditions hold
\begin{enumerate}
\item $p>2$,
\item $m\equiv n\pmod 2$ and $m\neq 2$, 
\item $(n,m)\neq (9,1)$,
\item $q^{\frac{n-m-1}2}>2m+1$, and
\item $\HH(q,n+m-1,n-m)$,
\end{enumerate}
or
\item[(ii)] $G=A_n$ and the following conditions hold
\begin{enumerate}
\item $p>3$,
\item $m\not\equiv n\pmod 2$ and $(n,p)=1$,
\item $(n,m)\neq (8,1),(9,0),(12,1),(24,1)$,
\item $q^{\frac{n-3m-1}4}>2m+1$, 
\item $m\ge 2$ or $\left(\frac q{n-m}\right)=1$, and
\item $\HH\left (q,\frac{n+m-1}2,n-m\right)$.
\end{enumerate}
\end{enumerate}
Then there exist $f,c\in\mathbb{F}_q[X]$ and $h\in\mathbb{F}_q[T]$
with ${\rm deg}(f)=n$, ${\rm deg}(c)=m$ and $(n-m)|{\rm deg}(h)$
such that the splitting field of $f-hc$ over $\mathbb{F}_q(T)$
is geometric with Galois group $G$
and is ramified over only one prime of $\mathbb{F}_q(T)$.
In particular, $\ram_{\mathbb{F}_q(T)}(G)=1$.
\end{theorem}

\begin{proof}
In $(i)$,
Theorem~\ref{thm:tame2}$(i)$
provides $f,c$ of the right degree such that the splitting field $K$ of 
$f-Uc$ over $\mathbb{F}_q(U)$ is geometric with Galois group $S_n$
and is ramified over only one finite prime
$\mathcal{F}$ of degree dividing $n+m-1$
and over the infinite prime, with ramification index $n-m$.
Thus $\HH(q,n+m-1,n-m)$ provides a suitable $h$ with $\mathcal{F}(h)$ irreducible,
hence Lemma~\ref{lem:disjoint} gives that 
the splitting field $K\F_q(T)$ of $f-hc$ over $\F_q(T)$ is geometric with Galois group $S_n$,
and Lemma~\ref{lem:eliminateinfty} gives that it is ramified only over $\mathcal{F}(h)$.

In $(ii)$,
Theorem~\ref{thm:tame2}$(ii)$ similarly
provides $f,c$ of the right degree such that the splitting field of 
$f-Uc$ over $\mathbb{F}_q(U)$ is geometric with Galois group $A_n$
and is ramified over only one finite prime
$\mathcal{F}$, with $\Disc_X(f-Uc)=\mathcal{F}^r$ of degree $n+m-1$,
and over the infinite prime, with ramification index $n-m$.
As $G=A_n$ implies that $\Disc_X(f-Uc)$ is a square, $r$ is even,
and so in particular ${\rm deg}(\mathcal{F})$ divides $\frac{n+m-1}{2}$.
Thus $\HH(q,\frac{n+m-1}2,n-m)$ 
provides a suitable $h$ with $\mathcal{F}(h)$ irreducible,
and we conclude again with Lemma~\ref{lem:disjoint}  and Lemma~\ref{lem:eliminateinfty}.
\end{proof}

\begin{remark}
Note that in the cases (2)-(5) of Proposition~\ref{lem:Hqdm},
one can obtain $h$ in Theorem~\ref{thm:tame} to be of degree {\em equal} to $n-m$.
\end{remark}

\subsection{Two ramified primes via small prime gaps}
\label{sec:twinprimes}
In this final subsection we present a different approach to produce extensions of group $S_n$ or $A_n$ with two ramified primes.
For this we will need the following result on primes with small gaps in $\F_q[T]$, which follows directly from \cite[Theorem 1.3]{CHLPT15}, which is a function field adaptation of the method of Maynard \cite{May15}. 

\begin{proposition}\label{thm:twinprimes} Assume $q\ge 107$ and fix $b\in\F_q^\times$. For all sufficiently large $d$ (in terms of $q$) there exist polynomials $h\in\F_q[T]$ of degree $d$ (not necessarily monic) such that $h$ and $h-b$ are both irreducible.
\end{proposition}

\begin{proof}
By \cite[Theorem 1.3]{CHLPT15} (applied with $m=2$ and $h_1,\dots,h_k$ an enumeration of $\F_q$) and the remark following it, as long as $q>k_0(2)=105$ (using the notation of the cited theorem) and $d\ge d_0(q)$ is sufficiently large, one can find $f\in\F_q[T],\deg f=d$ such that $f,f-a$ are irreducible for some $a\in\F_q^\times$ (the statement of the cited theorem only asserts the existence of infinitely many such $f$, but the proof shows that such an $f$ exists with any sufficiently large degree). Then $h=\frac ba f$ satisfies the assertion of the proposition.
\end{proof}

\begin{theorem}\label{thm:main2} 
Let $p$ be a prime number, $q$ a power of $p$, and $3\leq n<p$.
Then $\ram_{\mathbb{F}_q(T)}(S_n)\le 2$ and $\ram_{\mathbb{F}_q(T)}(A_n)\le 2$.
\end{theorem}

\begin{proof}
Consider the polynomial 
$$\label{eq:f1}
 f=X^n-X^{n-1}-U\in\F_q[U][X].
$$
Using (\ref{eq:disc2}) and $f'=nX^{n-2}(X-1+n^{-1})$ it is easy to show that $D(U):=\Disc_X(f)=aU^{n-2}(U-b)$ with 
$a\in\F_p^\times$ and $b=(1-n^{-1})^n-(1-n^{-1})^{n-1}\in\F_p^\times$.
Hence the splitting field of $f$ over $\Fb_q(U)$ is ramified over (at most) $0,b,\infty$.

The polynomial $f$ is irreducible over $\Fb_q(U)$ since it is primitive and linear in the variable $U$,
in particular $\Gal(f/\Fb_q(U))\leqslant S_n$ is transitive.
Let $\alpha$ be a root of $f$ in an algebraic closure of $\F_q(U)$. 
Then $U=(\alpha-1)\alpha^{n-1}$
and we see that in the extension $\Fb_q(\alpha)/\Fb_q(U)$ the point $U=0$ splits into an unramified point $\alpha=1$ and a point $\alpha=0$ with ramification index $n-1$, while $\infty$ is totally ramified. 
Hence, as $(n-1,p)=1$, by Lemma~\ref{lemcycle}(ii) the group $\Gal(f/\Fb_q(U))$ contains an $(n-1)$-cycle,
in particular it is primitive.
Since $U-b$ divides $D(U)$ exactly once,
by Lemma~\ref{lem:Dedekind} the prime $U-b$ splits in the extension $\Fb_q(\alpha)/\Fb_q(U)$ into $n-2$ unramified primes and one prime with ramification index 2, hence $\Gal(f/\Fb_q(U))$ contains a transposition by Lemma~\ref{lemcycle}(ii). 
Thus by Lemma~\ref{lem:primtrans}, $\Gal(f/\Fb_q(U))=S_n$.
Since $\Gal(f/\Fb_q(U))\leqslant\Gal(f/\F_q(U))\leqslant S_n$ we must have $\Gal(f/\F_q(U))=\Gal(f/\Fb_q(U))=S_n$.
Therefore the splitting field $K$ of $f$ over $\mathbb{F}_q(U)$ is geometric with Galois group $S_n$ and ramified over 
infinity and at most the two finite primes $U$ and $U-b$.
We now fix some $h\in\F_q[T]$ with $n|\deg h$ such that $h$ and $h-b$ are both irreducible --
for $q\ge 107$ such an $h$ is provided by
Proposition~\ref{thm:twinprimes},
and for $q<107$, a quick computer search 
(that only needs to go through the finitely many pairs $(n,q)$ with $3\leq n<p\leq q<107$)
verifies that there always exists such an $h$
even with $\deg h=n$.
Identify $U=h(T)$ and let $L=K\F_q(T)$ be the splitting field of $f(h(T),X)$ over $\F_q(T)$.
Lemma~\ref{lem:disjoint} gives that 
$L/\F_q(T)$ is geometric with Galois group $S_n$,
and Lemma~\ref{lem:eliminateinfty} gives that it is ramified only over the primes $h$ and $h-b$.
In particular, $\ram_{\mathbb{F}_q(T)}(S_n)\le 2$.

Now viewing $A_n$ as a subgroup of $\Gal(K/\F_q(U))$ consider the fixed field $F=K^{A_n}$. 
The extension $F/\F_q(U)$ is quadratic and ramified over at most three points $0,b,\infty$. 
Explicitly $F=\F_q(U,\sqrt{D(U)})$ (cf.~Lemma~\ref{lem:discgalois}) and 
since $U-b$ divides $D(U)$ exactly once, by Lemma~\ref{lem:Dedekind} we see that $U-b$ is ramified  in $F/\F_q(U)$. 
On the other hand the unique prime of $F$ lying over $U-b$ is unramified in the extension $K/F$ since the ramification index of $U-b$ in 
$K/\F_q(U)$ is 2. 
By the Riemann-Hurwitz formula we see that in fact $F/\F_q(U)$ is ramified over exactly one of $0,\infty$ (say $\infty$, the other case is similar) and $F$ has genus 0, hence $F=\F_q(U')$ for some $U'\in K$ \cite[Corollary 5.1.12]{Stichtenoth}. 
We have $\Gal(K/F)=A_n$ and the only primes of $F$ that ramify in $K/F$ are the primes lying over $U=0,\infty$. 
If 0 is inert in $F/\F_q(U)$ then only two primes of $F$ ramify in $K$, and therefore $\ram_{\mathbb{F}_q(U)}(A_n)=2$. 
Otherwise three primes of $F$ ramify in $K$, all of degree one, and applying a linear fractional transformation we may assume that these are $U'=0,1,\infty$. In this case we repeat the argument we used above for $S_n$ (taking $b=1$) to conclude that 
$\ram_{\mathbb{F}_q(T)}(A_n)\le 2$. 
\end{proof}

\section{$S_n$ and $A_n$ when $p\leq n$}
\label{sec:pleqn}

\noindent
Let $p$ be a prime number and $q$ a power of $p$.
The main goal of this section is to prove Theorem~\ref{thm:main1}, which contains Theorems \ref{thm:Sn}(2a), \ref{thm:An}(2a) and \ref{thm:An}(2b) as special cases.
Theorem~\ref{thm:main1} will follow from Theorem~\ref{thmlist} and Proposition~\ref{thmprod} which will appear below. The proof will be given in Section \ref{sec:thmprod}.

\begin{remark} The proof of Theorem~\ref{thm:main1} (and Theorem~\ref{thmlist} from which it follows) uses the Classification of Finite Simple Groups (via Theorem~\ref{thm:jones}) to handle the case $n=p+2,p>2$ (for both $A_n$ and $S_n$). In all other cases only elementary group theory (mainly Jordan's theorem on primitive permutation groups containing a prime cycle) is used.\end{remark}

\newcommand{\Lgt}{L_1^{\mathrm{gt}}}

\subsection{$\Lgt(q)$-realizable groups}

\begin{definition}\label{def:quasip} 
A finite group $G$ is called \emph{quasi}-$p$ if it is generated by its $p$-Sylow subgroups,
i.e.~$G=p(G)$. 
A finite group $G$ is called \emph{cyclic-by-quasi-$p$} if it is an extension of a cyclic group by a quasi-$p$ group,
i.e.~$d(G/p(G))\leq1$.
\end{definition}

\begin{example} \label{ex:cyclicbyquasip}
The groups $S_n$ and $A_n$ are cyclic-by-quasi-$p$ if $n\ge p$.
\end{example}

\begin{definition} 
We say that a finite group $G$ is \emph{$\Lgt(q)$-realizable} if there exists a geometric Galois extension $L/\F_q(T)$ with $\Gal(L/\F_q(T))=G$ that is ramified over at most two primes of $\F_q(T)$, both of degree one, and with wild ramification over at most one of them. 
An extension $L$ with this property is called an $\Lgt(q)$-\emph{realization} of $G$.\footnote{Abhyankar introduced the notation $L_1$ for the once-punctured affine line. The letters $\mathrm{gt}$ remind us that (up to M\"obius transformations) we only consider geometric Galois $\F_q$-covers of $L_1$ which are tamely ramified at $\infty$.}\end{definition}
Note that an $\Lgt(q)$-realizable group is automatically $\Lgt(q^r)$-realizable for any $r\ge 1$.

\begin{lemma}\label{lemreal}
Let $G$ be an $\Lgt(q)$-realizable group.
Then $G$ is cyclic-by-quasi-$p$ and satisfies $[G:p(G)]\mid q-1$,
and there exists an $\Lgt(q)$-realization of $p(G)$ that is ramified over a single prime of degree one.
\end{lemma}

\begin{proof}  
Let $L/\F_q(T)$ be an $\Lgt(q)$-realization of $G$.
We may assume that it is ramified only over $T=0,\infty$, tamely ramified over $\infty$. 
Let $H=p(G)\unlhd G$  and denote $K=L^H$. 
We have $\Gal(K/\F_q(T))=C:=G/H$. Since $(|C|,p)=1$ the extension $K/\F_q(T)$ is tamely ramified and by assumption it is
ramified at most over $T=0,\infty$,
so it must be of the form $K=\F_q(T^{1/n})$ with $n=|C|$ and $n|q-1$
(e.g.~since $K(T^{1/n})/\F_q(T^{1/n})$ is geometric, and unramified by Lemma~\ref{lem:abhyankar}, hence trivial).
In particular, $C=\Gal\left(\F_q(T^{1/n})/\F_q(T)\right)$ is cyclic, so $G$ is cyclic-by-quasi-$p$ and $[G:p(G)]=n\mid q-1$.

Let $U=T^{1/n}$ and let $e$ be the ramification index of the prime $U=\infty$ in $L/\F_q(U)$.
Note that $(e,p)=1$
and denote $K'=\F_q(S)$, where $S^e=U$. 
The extension $K'\Fb_q/\Fb_q(U)$ is Galois with $\Gal(K'\Fb_q/\Fb_q(U))=\Z/e\Z$ and since $L/\F_q(U)$ is geometric we have 
that $\Gal(L\Fb_q/\Fb_q(U))=p(G)$. 
Since $p(G)$ is quasi-$p$, the groups $p(G)$ and $\Z/e\Z$ have no nontrivial common quotients and therefore the extensions 
$K'\Fb_q,L\Fb_q$ of $\Fb_q(U)$ are linearly disjoint. 
Consequently the geometric extensions $K',L$ of $\F_q(U)$ are also linearly disjoint and we have $\Gal(L\F_q(S)/\F_q(S))=\Gal(L/\F_q(U))=p(G)$. 
By Lemma~\ref{lem:eliminateinfty}, $L\F_q(S)/\F_q(S)$ is ramified only over the prime $S$.
\end{proof}

\begin{question}\label{ques:Abhyankar}
Is a cyclic-by-quasi-$p$ group $G$ with $[G:p(G)]\mid q-1$ always $\Lgt(q)$-realizable?
\end{question}

This question is a variant of the arithmetic Abhyankar conjecture (Conjecture \ref{conj:abhyankar}) and is suggested by Abhyankar's general philosophy regarding ramified covers in characteristic $p$. 
We remark that Harbater and van der Put \cite[Theorem 5.5]{HavdP02} have given a non-obvious necessary condition for a group to be the Galois group of an extension of $\F_q(T)$ ramified over two geometric points, but this condition is automatically satisfied for cyclic-by-quasi-$p$ groups, 
as can easily be shown using Frattini's argument. 
The next theorem answers Question \ref{ques:Abhyankar} affirmatively for the symmetric and alternating groups (with some mild restrictions on $n,p$).

\begin{theorem}\label{thmlist} Let $p$ be a prime, $q$ a power of $p$, and $n\ge p$. The following groups are $\Lgt(q)$-realizable:
\begin{enumerate} 
\item $S_n$ if $n\neq p+1$ or $p=2$.
\item $A_n$ if $p>2$ and either $n\neq p+1$ or $\F_q\supseteq\F_{p^2}$ or $p=3$.
\item $A_n$ if $p=2$ and either $\F_q\supseteq\F_4$ or $10\neq n\ge 8$ and $n\equiv 0,1,2,6,7\pmod 8$.
\end{enumerate}
\end{theorem}

Theorem~\ref{thmlist} will be proved in Section \ref{secthmlist}. 
Note that in the case of $A_n,n\ge p>2$ Theorem~\ref{thmlist} combined with Lemma~\ref{lemreal}
and Example \ref{ex:cyclicbyquasip} immediately imply Theorem~\ref{thm:abhyankar}.

\subsection{Application to the minimal ramification problem}
\label{sec:thmprod}

We note that $\Lgt(q)$-re\-al\-iz\-abil\-ity implies
a positive answer to the minimal ramification problem:

%

\begin{lemma} \label{removeinf} \label{lemquasip} 
Let $G$ be an $\Lgt(q)$-realizable nontrivial finite group. 
Then for infinitely many primes $h\in\F_q[T]$ there exists a geometric Galois extension $L/\F_q(T)$ with $\Gal(L/\F_q(T))=G$ which is ramified only over $h$. 
In particular, $\ram_{\mathbb{F}_q(T)}(G)=1$ and $G$ satisfies Conjecture~\ref{bm}.
\end{lemma}

\begin{proof} 
Let $M/\F_q(U)$ be a geometric realization of $G$ ramified over $U=0$ and (at most) tamely ramified over $U=\infty$. 
Denote by $e$ the ramification index over $\infty$,
which by assumption is not divisible by $p$.
Let $h\in\F_q[T]$ irreducible of degree divisible by $e$
and identify $U=h(T)$.
By Lemma~\ref{lem:disjoint}, $M$ and $\F_q(T)$ are linearly disjoint over $\F_q(U)$,
and $M\F_q(T)/\F_q(T)$ is ramified only over $h$ by Lemma~\ref{lem:eliminateinfty}.
Therefore, $M\F_q(T)/\F_q(T)$ is geometric Galois of Galois group $G$ and ramified only over the single prime $h$,
hence $\ram_{\mathbb{F}_q(T)}(G)=1$.
\end{proof}

In this subsection we 
extend this to (certain) products of $\Lgt(q)$-realizable groups,
which we then combine with Theorem~\ref{thmlist} to prove Theorem~\ref{thm:main1}:

\begin{proposition}\label{thmprod}
 Let $G=G_1\times\ldots\times G_m$ be a product with each $G_i$ being $\Lgt(q)$-realizable. 
 Assume that there is a prime number $\ell$ such that for each $1\le i\le m$ either $G_i$ is quasi-$p$ or $\ell\mid [G_i:p(G_i)]$. Then Conjecture \ref{bm} holds for $G$ over $\F_q(T)$.
\end{proposition}

By a $k$-{\em variety} (for a field $k$) we mean a reduced $k$-scheme of finite type.
In what follows all maps are morphisms of $\F_q$-varieties,
and $\A^n$ (resp. $\P^n$) always denotes the affine (resp. projective) $n$-space over $\F_q$ considered as an $\F_q$-variety.
The function field of an irreducible $\F_q$-variety $X$ is denoted by $\F_q(X)$. 
We will also use the notion of a Galois cover of $\F_q$-varieties and its Galois group, see \cite[\S I.5]{Mil80} for the definition and basic properties. 
In particular, if $w\colon Y\rightarrow X$ is a Galois cover of $\F_q$-varieties, 
then $Y$ and $X$ are connected and $w$ is \'etale, and we denote its Galois group by $\Gal(Y/X)$. 
We call $w$ a \emph{geometric} Galois cover if $Y\times\mathrm{Spec}\,\Fb_q\to X\times\mathrm{Spec}\,\Fb_q$ is also a Galois cover (if $X$ is geometrically irreducible this is equivalent to $Y$ being geometrically irreducible).

\begin{proposition}\label{hit}
Let $U\subseteq\A^m$ be an open $\F_q$-subvariety and $Y\to U$ a geometric Galois cover of $\F_q$-varieties with $\Gal(Y/U)=G$.
Then 
there exists a morphism 
$\psi\colon\A^1\to\A^m$ 
 such that $\psi^{-1}(U)\neq\emptyset$ and $Y\times_U\psi^{-1}(U)\to\psi^{-1}(U)$ is a geometric Galois cover with Galois group $G$.\end{proposition}

\begin{proof} 
Denote $\nu=|G|$ and $Y'=Y\times_{\F_q}\mathrm{Spec}\,\F_{q^\nu},U'=U\times_{\F_q}\mathrm{Spec}\,\F_{q^\nu}$. 
Since $Y\rightarrow U$ is geometric, $Y'$ is irreducible and the cover $Y'=Y\times_UU'\to U$ is Galois with $\Gal(Y'/U)=G\times \Gal(\F_{q^\nu}/\F_q)$. (Generally, if $Z\rightarrow X$ and $W\rightarrow X$ are Galois covers then $Z\times_X W$ is a union of isomorphic Galois covers with Galois group a subgroup of $\Gal(Z/X)\times\Gal(W/X)$,
 and if $Z\times_X W$ is irreducible we have $\Gal(Z\times_X W/X)\cong \Gal(Z/X)\times\Gal(W/X)$. 
This fact can be deduced from \cite[Theorem 5.3]{Mil80} or by comparing with the Galois groups of the corresponding function field extensions.)
We will show that 
one can choose $h_1,\dots,h_m\in\F_q[T]$ such that the corresponding morphism $\psi\colon\mathbb{A}^1\rightarrow\mathbb{A}^m$
satisfies $\psi^{-1}(U)\neq\emptyset$ and
$Y'\times_U\psi^{-1}(U)\to\psi^{-1}(U)$ is a Galois cover with Galois group 
$\Gal(Y'/U)=G\times\Gal(\F_{q^\nu}/\F_q)$,
which will imply that
$Y\times_U\psi^{-1}(U)\to\psi^{-1}(U)$ is a {\em geometric} Galois cover with Galois group $G$,
since any extension of the field of constants occurring in this cover must be of degree dividing $\nu=|G|$.

First note that 
$Y'\times_U\psi^{-1}(U)\to\psi^{-1}(U)$ is again \'etale
for any $\psi$,
so it suffices to show that 
$h_1,\dots,h_m$ can be chosen so that
it is a Galois cover with Galois group $\Gal(Y'/U)$,
for which we are allowed to replace $U$ and $Y'$ by dense open subvarieties.
So since an \'etale morphism is locally standard \'etale (see e.g.~\cite[Proposition I.3.19]{Mil80}),
we can assume that $U=\{g(T_1,\ldots,T_m)\neq 0\}$ with $0\neq g\in\F_q[T_1,\ldots,T_m]$ and
$$
 Y'=\{f(T_1,\ldots,T_m,X)=0,g(T_1,\ldots,T_m)\neq 0\}\subseteq\A^{m+1}
$$ 
with $f\in\F_q[T_1,\ldots,T_m,X]$, 
and $Y'\to U$ is the projection to the first $m$ coordinates.
By \cite[Theorem 13.3.5 and Proposition 16.1.5]{FJ} there exist
$h_1,\ldots,h_m\in\F_q[T]$ with
$g(h_1,\ldots,h_m)\neq 0$, $f(h_1,\dots,h_m,X)\in\F_q(T)[X]$ is irreducible and 
$$\label{eq:galproduct}
 \Gal(f(h_1,\ldots,h_m,X)/\F_q(T))=\Gal(f(T_1,\ldots,T_m,X)/\F_q(T_1,\ldots,T_m))=\Gal(Y'/U).
$$ 
The corresponding $\psi$ satisfies
$\psi^{-1}(U)\neq\emptyset$, $Y'\times_U\psi^{-1}(U)$ is irreducible, and
$$
 \Gal(Y'\times_U\psi^{-1}(U)/\psi^{-1}(U))=\Gal(f(h_1,\ldots,h_m,X)/\F_q(T))=\Gal(Y'/U),
$$ 
concluding the proof.
\end{proof}

The next proposition combined with Lemma $\ref{lemquasip}$ implies Proposition~\ref{thmprod} in the special case when all $G_i$ are quasi-$p$.

\begin{proposition}\label{propprod} 
Let $G_1,G_2$ be $\Lgt(q)$-realizable groups with $G_2$ quasi-$p$. Then $G_1\times G_2$ is $\Lgt(q)$-realizable.
\end{proposition}

\begin{proof} 
First assume that both $G_1,G_2$ are quasi-$p$. 
Let $T_1,T_2$ be independent variables and let $L_i/\F_q(T_i)$ be an $\Lgt(q)$-realization of $G_i$ ramified only over $\infty$ 
(we may assume this by Lemma~\ref{lemreal}). 
It corresponds to a geometric Galois cover $Y_i\to\A^1=\Spec(\F_q[T_i])$ with $\Gal(Y_i/\A^1)=G_i$. 
Taking the product of these covers we obtain a geometric Galois cover $Y=Y_1\times Y_2\to \A^2$ with $\Gal(Y/\A^2)=G_1\times G_2$.
By Proposition~\ref{hit} there exists a morphism $h\colon\A^1\to \A^2$ such that $Z=Y\times_{\A^2}\A^1\to\A^1$ 
is a geometric Galois cover with $\Gal(Z/\A^1)=\Gal(Y/\A^2)=G$. 
The extension $L=\F_q(Z)$ of $\F_q(\A^1)$ now gives a geometric realization of $G$ ramified only over $\infty$.

Now let $G_1,G_2$ be $\Lgt(q)$-realizable with only $G_2$ assumed quasi-$p$. Iterating the above claim we see that an arbitrary product of $\Lgt(q)$-realizable quasi-$p$ groups is again $\Lgt(q)$-realizable (and of course quasi-$p$ as well). 
In particular since $G_2$ is quasi-$p$ and $\Lgt(q)$-realizable, for every $m$ we can find a geometric Galois extension $L/\F_q(T)$ with Galois group $(G_2)^m$ ramified only over $\infty$, and it contains linearly disjoint subextensions $L_1,\ldots,L_m$ such that $\Gal(L_i/\F_q(T))=G_2$. 
Let $L/\F_q(T)$ be a geometric realization of $G_1$ ramified over $\infty$, tamely ramified over 0 and unramified elsewhere. 
For $m$ sufficiently large, one of the $L_i$ is linearly disjoint from $L$ over $\F_q(T)$, and then 
$K=LL_i$ is a geometric Galois extension of $\F_q(T)$ with $\Gal(K/\F_q(T))=G_1\times G_2$, ramified only over $0,\infty$, tamely over $0$.
\end{proof}
%

\begin{proof}[Proof of Proposition~\ref{thmprod}] 
Let $G=G_1\times\dots\times G_m$ be a product of $\Lgt(q)$-realizable groups.
If $G$ is quasi-$p$, then by Proposition~\ref{propprod} the group $G$ is $\Lgt(q)$-realizable and by Lemma~\ref{lemquasip} it satisfies Conjecture \ref{bm}.

Otherwise we can use Proposition~\ref{propprod} to absorb all the quasi-$p$ factors into one of the non-quasi-$p$ factors and reduce to the case when none of the $G_i$ are quasi-$p$. Assuming none of the $G_i$ are quasi-$p$ we have $d\left((G/p(G))^\ab\right)=m$,
since by Lemma~\ref{lemreal} the $G_i$ are cyclic-by-quasi-$p$ and the assumptions of the proposition imply that each $G_i$ has a quotient isomorphic to $\Z/\ell\Z$. 
By Lemma~\ref{removeinf}, for each $G_i$ there exists a geometric Galois extension $L_i/\F_q(T)$ with Galois group $G_i$ branched only over a finite prime $h_i\in\F_q[T]$ and we may take the $h_i,i=1,\ldots,m$ to be pairwise distinct. 
Then the extensions $L_i/\F_q(T)$ are pairwise linearly disjoint, since their branch loci are disjoint and $\F_q(T)$ has no unramified geometric extensions. 
Taking $L=L_1\cdots L_m$ we see that $\Gal(L/\F_q(T))=G$ and $L/\F_q(T)$ is ramified exactly over $h_1,\ldots,h_m$, so $\ram_{\mathbb{F}_q(T)}(G)\le m$.
Since by 
Corollary \ref{cor:lowerbound}
we also have $\ram_{\mathbb{F}_q(T)}(G)\geq m$,
we conclude that $\ram_{\mathbb{F}_q(T)}(G)=m$,
which is what Conjecture \ref{conj} predicts.
\end{proof}

\begin{proof}[Proof of Theorem~\ref{thm:main1}] 
Let $G=G_1\times\ldots\times G_m$ with each $G_i$ a symmetric or alternating group satisfying the assumptions of Theorem~\ref{thm:main1}. 
Since Theorem~\ref{thmlist} has the same assumptions, each $G_i$ is $\Lgt(q)$-realizable. 
If $p>2$ we may take $\ell=2$ and then each $G_i$ is either an alternating group and thus quasi-$p$ or is a symmetric group and then $\ell\mid 2=[G_i:p(G_i)]$. 
Thus the conditions of Proposition~\ref{thmprod} are satisfied and the conclusion follows. 
If $p=2$ then we take $\ell=3$ and $G_i$ is quasi-2 unless $G_i=A_3,A_4$ in which case $\ell\mid 3=[G_i:p(G_i)]$ and again Proposition~\ref{thmprod} applies and the conclusion follows.
\end{proof}

\subsection{Proof of Theorem~\ref{thmlist}}\label{secthmlist}

Let $n\ge p$.
In the present subsection we prove for $G\in\{S_n,A_n\}$
and each of the pairs $(n,q)$ from Theorem~\ref{thmlist}
that the cyclic-by-quasi-$p$ group $G$ is $\Lgt(q)$-realizable. 
We construct realizations for these groups by writing down explicit equations. 
Our constructions are inspired by the work of Abhyankar \cite{Abh92},
 and in the case $p=2$ we directly use the results of Abhyankar, Ou, Sathaye and Yie \cite{AOS94,AbYi94}. Some of the Galois group calculations rely on the classification of multiply transitive groups, which in turn relies on the Classification of Finite Simple Groups (CFSG). We will indicate when CFSG and its applications are used. 
In all cases we will write an explicit polynomial $f\in\F_q(T)[X]$ of degree $n$ and show that $\Gal(f/\F_q(T))=G$ and that the splitting field of $f$ over $\F_q(T)$ is geometric 
and is ramified over a single prime of degree one if $G$ is quasi-$p$ and over two primes of degree one with tame ramification over at least one of them if $G$ is cyclic-by-quasi-$p$ but not quasi-$p$. Note that we always have the inclusions 
$\Gal(f/\Fb_q(T))\leqslant\Gal(f/\F_q(T))\leqslant S_n$,
so for $G=S_n$, proving that $\Gal(f/\Fb_q(T))=S_n$ shows both that $\Gal(f/\F_q(T))=G$ and that the splitting field of $f$ is geometric,
and for $G=A_n$, proving that $\Gal(f/\Fb_q(T))\geqslant A_n$ and $\Gal(f/\F_q(T))\leqslant A_n$ shows
that $\Gal(f/\F_q(T))=G$ and that the splitting field of $f$ is geometric.
Recall from Lemma~\ref{lem:discgalois}  that if $p>2$ then $\Gal(f/\F_q(T))\leqslant A_n$
if and only if $\Delta(f)$ is a square in $\F_q(T)$
(where we always take the discriminant with respect to $X$).
We distinguish four main cases according to $G=S_n$ or $A_n$ and $p=2$ or $\neq 2$,
and several subcases according to $n$ and $q$.

\maincase{$G=S_n,p>2$}\label{sec:sn}
In this case $G$ is cyclic-by-quasi-$p$ but not quasi-$p$, so we are looking for a polynomial $f\in\F_q(T)[X],\deg f=n$ with 
$\Gal(f/\Fb_q(T))=S_n$ with splitting field ramified exactly over $T=0,\infty$, tamely over one of these points. 

\case{$n$ odd, $n>p, p\nmid n+1$} 
Let $g=X^{n+1}-(1+T)X^p+T\in\mathbb{F}_p[T,X]$ and
\begin{equation}\label{eq:snodd} 
 f=\frac{g}{X-1}=X^p\frac{X^{n+1-p}-1}{X-1}-(X-1)^{p-1}T\in\F_p[T,X].
\end{equation}
Note that $f$ is a monic polynomial of degree $n$ in $X$. It is irreducible over $\Fb_q(T)$ because it is linear in the variable $T$ with coprime coefficients $X^p\frac{X^{n+1-p}-1}{X-1}$ and $-(X-1)^{p-1}$ (here we use the assumption $p\nmid n+1$).
Since $g'=(n+1)X^n$, using (\ref{eq:discprod}),(\ref{eq:disc2}) and (\ref{eq:resdefinition}) we have 
$$
 \Disc(f)=\Disc(g)/\Res(X-1,f)^2=\pm (n+1)^{n+1}T^n/f(1)^2=\pm (n+1)^{n-1}T^n.
$$
Let $\alpha$ be a root of $f$ in an algebraic closure of $\F_q(T)$. 
Since $\Disc(f)=\pm (n+1)^{n-1}T^n$, the extension $\Fb_q(T,\alpha)/\Fb_q(T)$ (and therefore its normal closure) is ramified at most over $0,\infty$. We now compute the ramification indices of this extension over 0. To this end observe that
by (\ref{eq:snodd}) we have 
\begin{equation}\label{eq:snoddt} 
 T=\frac{\alpha^p(\frac{\alpha^{n+1-p}-1}{\alpha-1})}{(\alpha-1)^{p-1}}
\end{equation} 
and we see that the divisor of zeros of $T$ over $\Fb_q(T,\alpha)=\Fb_q(\alpha)$ is composed of one prime of multiplicity $p$ and other primes of multiplicity 1. Therefore the ramification indices over $T=0$ are $p,1,\ldots,1$, so by Lemma~\ref{lemcycle}(ii) $G=\Gal(f/\Fb_q(T))$ contains a cycle of length $p$. Note also that by (\ref{eq:snoddt}), the ramification indices over $T=\infty$ are $p-1,n+1-p$,
in particular the ramification over $T=\infty$ is tame.
Therefore the action of $\Gal(f/\Fb_q(T))$ on the roots of $f$ is primitive by Lemma~\ref{lem:primitive},
so since it
contains a cycle of length $p$ we can apply 
Theorem~\ref{thm:jones} with $\ell=p$. 

If $n>p+2$ then Jordan's Theorem (see Remark 
\ref{remark:jordan}) implies that $\Gal(f/\Fb_q(T))\geqslant A_n$ and since $\Disc(f)$ is not a square 
we have $\Gal(f/\F_q(T))=\Gal(f/\Fb_q(T))=S_n$ and the polynomial $f
$ satisfies all the required properties. 

If $n=p+2$ then we use Theorem~\ref{thm:jones}(ii) to 
conclude that either $\Gal(f/\Fb_q(T))\geqslant A_n$ (and then we argue as before) or $p=2^k-1$ is a 
Mersenne prime and ${PGL}_2(2^k)\leqslant \Gal(f/\Fb_q(T))\leqslant{P\Gamma L}_2(2^k)$ with its action 
on the roots being the standard action of such a group on $\P^1(\F_{2^k})$. 
However, 
the ramification indices over $\infty$ are $3,p-1$ and since $p>3$ (as $p\nmid n+1=p+3$),
 by Lemma  \ref{lemcycle}(i), $\Gal(f/\Fb_q(T))$ contains a permutation with cycle structure $3,p-1$. The group 
$P\Gamma L_2(2^k)$ with its standard action on $\P^1(\F_{2^k})$ has no such element (see 
\cite[Corollary 4.6]{GMPS16}), a contradiction.

\begin{remark}
Note that in treating the case $n=p+2$ we made use of Theorem~\ref{thm:jones}(ii), which makes use of the CFSG.
\end{remark}

\case{$n=p$} Let 
$$
 f=X^p+X^2-T\in\F_p[T,X].
$$ 
We have $f'=2X$ and therefore by (\ref{eq:disc2}) we have $\Disc(f)=aT,a\in\F_p^\times$, hence the splitting field of $f$ over $\Fb_q(T)$ is ramified at most over 0 and $\infty$. 
The polynomial $f$ is irreducible over $\Fb_q(T)$ since it is monic (up to sign) and linear in the variable $T$.

Let $\alpha$ be a root of $f$ in an algebraic closure of $\F_q(T)$. Then $T=\alpha^2(\alpha^{p-2}+1)
$ and we get (reasoning as in the previous case) that in the extension $\Fb_q(\alpha)=\Fb_q(T,\alpha)/
\Fb_q(T)$ the prime $T=0$ splits into $p-2$ unramified primes and the prime $\alpha=0$ with 
ramification index 2. 
In particular, the splitting field of $f$ is tamely ramified over $T=0$
and $\Gal(f/\Fb_q(T))\leqslant S_n$ contains a transposition by Lemma~\ref{lemcycle}. 
On the other hand $T=\infty$ is totally ramified with ramification index $p$. 
By Lemma~\ref{lem:primitive}, $\Gal(f/\Fb_q(T))$ is primitive,
so by Lemma~\ref{lem:primtrans} we 
have $\Gal(f/\Fb_q(T))=S_n$ and $f$ is as desired.

\case{$n$ odd, $p|n+1$, $n\neq 2p-1$}
Note that these conditions imply $n\ge 4p-1$. 
Choose $h\in\F_p[X],\deg h=3$ monic irreducible.
There exists $u\in\F_p[X],\deg u=3$ with $X^{n+3-p}\equiv u^p\pmod h$,
since $a\mapsto a^p$ 
is an automorphism of $\F_p[X]/h$ and we can
assume without loss of generality that $X\nmid u$ by adding $\pm h$ if necessary.
Let $g=X^{n+3}-u(X)^pX^p-h(X)^pT$ and
\begin{equation}\label{eq:snodd1} 
 f=\frac{g}{h}=\frac{X^{n+3-p}-u(X)^p}{h(X)}X^p-h(X)^{p-1}T\in\F_p[T,X].
\end{equation}
The polynomial $f$ is monic in $X$ with $\deg_X f=n$.
Let $v=X^{n+3-p}-u(X)^p$ and note that $v'=(n+3)X^{n+2-p}\neq 0$, 
in particular $v$ is separable (as $X\nmid u$),
which also implies that $h^2\nmid v$ and so $f$ 
is linear in $T$ with coprime coefficients, hence irreducible.
Using (\ref{eq:discprod}),(\ref{eq:disc2}) and (\ref{eq:resdefinition}) we compute
$$
 \Disc(f)=\Disc(g)/(\Res(h,f)^2\Disc(h))=aT^{n+2}
$$ for some $a\in\F_p^\times$,
as $\Res(h,f)$ is independent of $T$ and non-zero since $h^2\nmid v$.
Consequently the splitting field of $f$ 
is ramified only over $T=0,\infty$. 
Since $n$ is odd, $\Disc(f)$ is not a square in $\Fb_q(T)$ and so $\Gal(f/\Fb_q(T))\not\leqslant A_n$. 

Let $\alpha$ be a root of $f$ in an algebraic closure of $\F_q(T)$. By (\ref{eq:snodd1}) we have
\begin{equation}\label{eq:snodd1t} 
 T=\frac{\left(\frac{\alpha^{n+3-p}-u(\alpha)^p}{h(\alpha)}\right)\alpha^p}{h(\alpha)^{p-1}}
 \end{equation}
and therefore $\Fb_q(T,\alpha)=\Fb_q(\alpha)$. Now since $n\ge 4p-1$ we see from (\ref{eq:snodd1t}) that the 
primes of $\Fb_q(\alpha)$ over $T=0$ are $\alpha=0$ with ramification index $p$ and the roots of $v/h$ with ramification index 1,
so by Lemma~\ref{lemcycle} the group $\Gal(f/\Fb_q(T))$  contains a cycle of length $p$.
Similarly, the primes of $\Fb_q(\alpha)$ over $T=\infty$ are 
the 3 roots of $h$ with multiplicity $p-1$ each, and $\alpha=\infty$ with multiplicity $n-3(p-1)\equiv 2\pmod p$,
in particular the splitting field of $f$ is tamely ramified over $T=\infty$.
By Lemma~\ref{lem:primitive},
 $\Gal(f/\Fb_q(T))$ acts primitively on the roots of $f$. 
 By Jordan's Theorem (Theorem~\ref{thm:jones}(i) in the elementary case $\ell=p\le n-3$ prime) we have $\Gal(f/\Fb_q(T))\geqslant A_n$ and therefore $\Gal(f/\Fb_q(T))=S_n$.

\case{\bf $n=2p-1$} 
Let 
$$
 f=\frac{X^{2p}-TX^p-X^2+T}{X-1}=\frac{X^{2p-2}-1}{X-1}X^2-T(X-1)^{p-1}\in\F_p[T,X].
$$ 
The polynomial $f$ is monic in $X$ and irreducible over $\Fb_q(T)$, since it is linear in $T$ with coprime coefficients. 
Using (\ref{eq:discprod}),(\ref{eq:disc2}) and (\ref{eq:resdefinition}) we compute
$$
 \Disc(f)=\Disc(X^{2p}-TX^p-X^2+T)/\Res(X-1,f)^2=aT
$$ 
for some $a\in\F_p^\times$. We see that the splitting field of $f$ over $\Fb_q(T)$ is ramified only over $0,\infty$. Let $\alpha$ be a root of $f$ in an algebraic closure of $\F_q(T)$. We have
$$T=\frac{\left(\frac{\alpha^{2p-2}-1}{\alpha-1}\right)\alpha^2}{(\alpha-1)^{p-1}},$$
from which we see that in the extension $\Fb_q(\alpha)=\Fb_q(T,\alpha)$ of $\Fb_q(T)$ there are $2p-2$ primes lying over $T=0$ 
with one of them having ramification index $2$ and the other unramified, 
while over $T=\infty$ we have the primes $\alpha=\infty$ with ramification index $p$, and $\alpha=1$ with ramification index $p-1$.
Therefore, 
by Lemma~\ref{lemcycle},
$\Gal(f/\Fb_q(T))\leqslant S_n$ 
contains both a transposition and cycle of length $p$.
The latter implies that it is primitive,
and so $\Gal(f/\Fb_q(T))=S_n$ by Lemma~\ref{lem:primtrans}. 

\case{$n$ even, $p\nmid n$, $n>p+1$} 
Let 
$$
 f=X^n+X^p-T\in\F_p[T,X].
$$  
Since $f$ is monic (up to sign) and linear in $T$, it is irreducible over $\Fb_q(T)$. 
Since $f'=nX^{n-1}$, by (\ref{eq:disc2}) we have $\Disc(f)=aT^{n-1},a\in\F_p^\times$. 
Therefore the splitting field of $f$ is only ramified over $T=0$ and $T=\infty$, and $\Disc(f)$ is not a square in $\Fb_q(T)$ (since $n$ is even) and so $\Gal(f/\Fb_q(T))\not\leqslant A_n$. 
Denoting by $\alpha$ a root of $f$ in an algebraic closure of $\F_q(T)$, we see that $T=\alpha^p(\alpha^{n-p}+1)$. 
The polynomial $X^{n-p}+1$ is separable by the assumption $p\nmid n$, and so the ramification indices of $\Fb_q(T,\alpha)=\Fb(\alpha)$ over $T=0$ are $p,1,\ldots,1$, and $T=\infty$ is totally ramified with ramification index $n$ (in particular tame). 
Therefore, $\Gal(f/\Fb_q(T))\leqslant S_n$ is primitive by Lemma~\ref{lem:primitive} and contains a cycle of length $p$ by Lemma~\ref{lemcycle},
so by Jordan's Theorem (i.e. Theorem~\ref{thm:jones}(i) in the case $\ell=p\le n-3$ prime)
we get that $\Gal(f/\Fb_q(T))=S_n$.

\case{$n$ even, $p|n$} 
By our assumptions, $n\ge 2p$.
Choose $h(X)\in\F_p[X]$ monic irreducible with $\deg h=2$ and $h'(0)\neq 0$ (which always exists).
Since $a\mapsto a^p$ is an automorphism of $\F_p[X]/h$, there exists $u(X)\in\F_p[X]$ with $\deg u<n/p$ and $u^p\equiv X^{n+2-p}\pmod h$.
We can assume without loss of generality that $u(0)\neq 0$,
since if $n>2p$ we can add $\pm h$ to $u$ if necessary, 
while if $n=2p$ and we take the unique $u$ with $\deg u\le 1$ such that $u^p\equiv X^{p+2} \pmod h$ then 
automatically $u(0)\neq 0$, since otherwise $u=cX,c\in\F_p^\times$ and then $h=X^2-c$, contradicting our assumption $h'(0)\neq 0$.
Let $g=X^{n+2}-u(X)^pX^p-Th(X)^p\in\F_p[T,X]$ and
\begin{equation}\label{eq:sneven}
f=\frac{g}{h}=\frac{X^{n+2-p}-u(X)^p}{h(X)}X^p-Th(X)^{p-1}\in\F_p[T,X].
\end{equation}
The polynomial $f$ is monic in $X$ of degree $n$. 
Let $v=X^{n+2-p}-u(X)^p$. 
As $v'=2X^{n+1-p}$ and $u(0)\neq 0$, we have 
that $v$ is separable, in particular $h^2\nmid v$.
So $f$ is linear in $T$ with coprime coefficients, hence irreducible over $\Fb_q(T)$. 
Using (\ref{eq:discprod}),(\ref{eq:disc2}) and (\ref{eq:resdefinition}), we compute
$$
 \Disc(f)=\Disc(g)/(\Res(h(X),f)^2\Delta(h))=aT^{n+1},a\in\F_p^\times
$$ 
and conclude that the splitting field of $f$ over $\Fb_q(T)$ is ramified at most over $T=0,\infty$ and that $\Gal(f/\Fb_q(T))\not\leqslant A_n$ (since $n$ is even).

Denoting by $\alpha$ a root of $f$ in an algebraic closure of $\F_q(T)$ we have (by (\ref{eq:sneven}))
$$
 T=\frac{\alpha^p\left(\frac{\alpha^{n+2-p}-u(\alpha)^p}{h(\alpha)}\right)}{h(\alpha)^{p-1}}.
$$
The polynomial $v$ is separable, hence the primes of $\Fb_q(T,\alpha)=\Fb_q(\alpha)$ over $T=0$ are $\alpha=0$ with ramification index $p$ and $n-p$ primes with ramification index 1. 
The primes over $T=\infty$ are the two roots of $h$ with ramification index $p-1$ and $\alpha=\infty$ with ramification index $n-2(p-1)\equiv 2\pmod p$,
in particular $T=\infty$ is tamely ramified.
Therefore again, $\Gal(f/\Fb_q(T))\leqslant S_n$ is primitive by Lemma~\ref{lem:primitive} and contains a cycle of length $p$ by Lemma~\ref{lemcycle},
so by Jordan's Theorem (i.e. Theorem~\ref{thm:jones}(i) in the case $\ell=p\le n-p\leq n-3$ prime)
we get that $\Gal(f/\Fb_q(T))=S_n$.

\maincase{$G=S_n,p=2$}
In this case $G$ is quasi-$p$. 

\case{$n$ odd} 
We have $n\ge 3$. 
Let
$$
 f(X)=X^n+TX^{n-2}+1\in\mathbb{F}_2[T,X].
$$ 
By \cite[\S 11.I.5]{Abh92} (with $t=n-2$, $a=1$, $s=1$), 
we have $\Gal(f/\Fb_2(T))=S_n$ and the splitting field of $f$ is ramified only over $T=\infty$. 

\case{$n$ even} 
Let
$$
 f(X)=\left((X+1)^{n-1}+X^{n-1}\right)(X+1)^2+T^{n-1}X^{n-1}\in\mathbb{F}_2[T,X].
$$ 
By \cite[\S 12.IV.4]{Abh92} (with $t=n-1$, $s=n-1$, $a=1$, $b=1$) 
we have $\Gal(f/\Fb_2(T))=S_n$ and the splitting field of $f$ is ramified only over $T=0$. 

\maincase{$G=A_n, p>2$}~\\

\case{$n\neq p+1$}
In this case we have shown that $S_n$ is $\Lgt(q)$-realizable, and since $p>2$ we have $p(S_n)=A_n$. Therefore by Lemma~\ref{lemreal}, the group $A_n=p(S_n)$ is also $\Lgt(q)$-realizable.

\case{$p=3,n=4$} 
Let
$$
 f=X^4-TX^3+1\in\mathbb{F}_3[T,X].
$$ 
As $f'=X^3$ and $f(0)=1$ we compute that $\Delta(f)=1$,
and thus the splitting field of $f$ is ramified only over $T=\infty$
and $\Gal(f/\F_q(T))\leqslant A_4$.
The group $\Gal(f/\Fb_q(T))\leqslant A_4$ is transitive on the roots (since $f$ is irreducible, being linear in the variable $T$) and has order divisible by $3$ (otherwise the splitting field of $f$ over $\Fb_q(T)$ would be tamely ramified and ramified only over infinity, which implies it is trivial). 
Thus $\Gal(f/\F_q(T))=\Gal(f/\Fb_q(T))=A_4$.

\case{$n=p+1,p>3, \F_q\supseteq\F_{p^2}$}\label{sec34} 
Let $s\ge 1$ and $2\le a\le p-1$ be integers and let
$$
 f=(X+1)\left(X+\frac{a}{a-1}\right)^p-T^sX^a\in\mathbb{F}_p[T,X].
$$ 
By \cite[\S 22]{Abh92} (with $\tau=a$, $Y=T^s$, $b=\frac{a}{a-1}$, but note the nonstandard sign convention in the definition of the discriminant),
\begin{equation}\label{eq:disc_mrt}
 \Disc(f)=(-1)^{(p+1)/2}\frac{a^{2a-1}}{(a-1)^{2a-3}}T^{s(p+1)}\in\mathbb{F}_p[T].
\end{equation}
It is shown in \cite[\S 12.IV.3]{Abh92} (with $t=a$, $b=\frac{a}{a-1}$) that if $p>5$, $2\le a\le \frac{p-1}2,(a,p+1)=1$ and $a(p+1-a)|s$, then $\Gal(f/\Fb_q(T))=A_n$ and the splitting field of $f$ is ramified only over $T=0$. 
If additionally $\F_q\supseteq\F_{p^2}$, then $\Disc(f)$ is a square in $\mathbb{F}_q[T]$ and we have $\Gal(f/\F_q(T))=A_{n}$. 
Note that for all $p>5$, an $a$ with $(a,p+1)=1$ and $a\not\equiv\pm 1\pmod{p+1}$ can be found 
(as $\phi(x)>2$ for $x>6$),
and after replacing $a$ with $p+1-a$ if necessary to assume that $2\leq a\leq\frac{p-1}{2}$,
we can set $s=a(p+1-a)$.

If $p=5,n=6,\F_q\supseteq\F_{25}$ we take the polynomial
$$
 f=(X+1)(X+2)^5-T^4X^2\in\F_p[T,X].
$$
The discriminant of $f$ is computed by (\ref{eq:disc_mrt}) with $a=2,s=4$ and equals $\Delta(f)=2T^{24}$, which is a square in $\F_{25}(T)$ and hence $\Gal(f/\F_q(T))\leqslant A_5$. 
On the other hand, by \cite[\S 12.IV.2]{Abh92} 
(with $t=2,b=2$), the splitting field of $f$ over $\Fb_5(T)$ is ramified only over $T=0$, and its Galois group is $A_5$, 
hence $\Gal(f/\F_q(T))=\Gal(f/\Fb_q(T))=A_5$.

\begin{remark} If we want to drop the assumption $\F_q\supseteq\F_{p^2}$ in the case $p>5$ we would need to find an $a$ with $2\le a\le \frac{p-1}2,(a,p+1)=1$ such that $(-1)^{(p+1)/2}a(a-1)$ is a square modulo $p$. It is not hard to show (using a P\'olya-Vinogradov-type inequality and an elementary sieve argument) that this is possible for all $p$ sufficiently large, and we conjecture that such an $a$ exists for all $p>13$. While proving this should be doable by means of a careful analysis and sufficiently large computer search, we did not pursue this.\end{remark}

\maincase{$G=A_n,p=2$}
The group $A_n$ is always cyclic-by-quasi-2 and it is quasi-2 iff $n\neq 3,4$. Most of the required realizations were constructed by Abhyankar, Ou, Sathaye and Yie \cite{Abh92, Abh93, AOS94,AbYi94}.

\case{$n=3,4$, $\F_q\supseteq\F_4$} 
For $A_3\cong\Z/3\Z$ we can take the extension $K=\F_q(s)/F_q(T)$ with $s^3=T$ which is Galois with group $A_3$ if $\F_q\supseteq\F_4$
and ramified only over $T=0,\infty$. 
Denote by $\zeta$ any element of $\F_4\setminus\F_2$, and consider the splitting field $L$ of the polynomial 
$$
 f=(X^2+X+s)(X^2+X+\zeta s)(X^2+X+\zeta^2 s)\in\mathbb{F}_2(s)[X]
$$
over $K$.  
The extension $L/\F_q(T)$ is Galois (since $s,\zeta s,\zeta^2 s$ are a Galois orbit over $\F_q(T)$) and by Artin-Schreier theory $\Gal(L/K)=\Z/2\Z\times\Z/2\Z$ with $\Gal(L/\F_q(T))$ acting nontrivially on the order 2 subgroups of $\Gal(L/K)$ by conjugation, in particular $\Gal(L/\F_q(T))$ is non-abelian.
The only non-abelian extension of $\Z/3\Z$ by $\Z/2\Z\times\Z/2\Z$ is $A_4$. 
Finally observe that $L/\F_q(T)$ is ramified only over $0,\infty$ and is therefore an $\Lgt(q)$-realization of $A_4$.

\begin{remark}
Note that by Lemma~\ref{lemreal}, both for $n=3$ and for $n=4$ the condition $\F_q\supseteq\F_4$ is necessary for $A_n$ to be $\Lgt(q)$-realizable since $A_3,A_4$ have cyclic quotients of order 3. 
\end{remark}

\case{$n=5$, $\F_q\supseteq\F_4$} 
Let
$$
 f=X^5+TX+1\in\F_2[T,X].
$$ 
It follows from \cite[\S 11.III.1]{Abh92} (with $q=4$, $t=s=1$, $a=-1$, in the notation used there) that $\Gal(f/\Fb_2(T))\cong PSL_2(4)\cong A_5$ and the splitting field of $f$ is ramified only over $T=\infty$. 
By \cite[2.23]{AOS94} (with $K=\F_q(T)$, $d=4$, $e=1$, $\bar{b}_d=T$, $\bar{b}_n=1$),
$\F_q\supseteq\F_4$ implies that $\Gal(f/\F_q(T))\leqslant A_5$.
Thus $\Gal(f/\F_q(T))=\Gal(f/\Fb_q(T))=A_5$.

\case{$n=6,7$, $\F_q\supseteq\F_4$} 
Consider the polynomials
\begin{eqnarray*}
f_6&=&X^6+T^{27}X^5+T^{54}X^4+(T^{18}+T^{36})X^3+T^{108}X^2+(T^{90}+T^{135})X+T^{162},\\
f_7&=&X^7+TX^4+X^2+1.
\end{eqnarray*}

The polynomials $f_6,f_7$ were found by Abhyankar and Yie \cite[Theorems 2.10 and 2.11]{AbYi94}, who showed that the splitting fields of $f_n,n=6,7$ are ramified only over $\infty$ and,
under the assumption $\F_q\supseteq\F_4$,
that $\Gal(f_n/\F_q(T))=\Gal(f_n/\Fb_q(T))=A_n$.

\case{$n\ge 9$ odd, $n\equiv 1,7\pmod 8$ or $\F_q\supseteq\F_4$} 
Let 
$$
 f=X^n+TX^{n-4}+1\in\mathbb{F}_2[T,X].
$$  
By \cite[Theorem 2]{Abh93} (with $t=n-4$, $q=4$ in the notation of the cited paper), 
the splitting field of $f$ is ramified only over $\infty$ and we have $\Gal(f/\F_q(T))\geqslant\Gal(f/\Fb_q(T))\geqslant A_n$. 
Conversely, by \cite[(2.27)]{AOS94} (with $t=n-4$, $b_{n-t}^*=T$, $b_n^*=1$) we have $\Gal(f/\F_q(T))\leqslant A_n$ if either $\F_q\supseteq\F_4$ or $n\equiv 1,7\pmod 8$. 
In these cases we conclude that $\Gal(f/\F_q(T))=\Gal(f/\Fb_2(T))=A_n$.

\case{$n\ge 8$ even, $10\neq n\equiv 0,2,6\pmod 8$ or $\F_q\supseteq\F_4$} 
Let $1\le t\le n$ with $(t,n)=1$, and
$$
 f=X^n+X^t+T^t\in\mathbb{F}_2[T,X].
$$ 
By \cite[\S 11.II.5]{Abh92}  (with $s=t$, $a=1$), if $2\le t\le n-4$ then $\Gal(f/\F_q(T))\geqslant\Gal(f/\Fb_2(T))\geqslant A_n$ and the splitting field of $f$ is ramified only over $\infty$. 
Conversely, by \cite[Theorem 2.27]{AOS94} (with $b_{n-t}^*=1$, $b_n^*=T^t$) we have $\Gal(f/\F_q(T))\leqslant A_n$ if either $\F_q\supseteq\F_4$ or $n\equiv 0\pmod 8$ or $n\equiv 2,6\pmod 8,2t\equiv n\pmod 8$. 
In each of these cases we can choose a suitable $t$:
For 
$n\geq 8$
there exists $2\leq t\leq n-4$ with $(t,n)=1$
since $\phi(n)>2$,
and for $n\equiv 2,6\pmod 8,n>10$,
the choice $t=\frac n2-4$ satisfies $2\le t\le n-4$, $(t,n)=1$ and $2t\equiv n\pmod 8$.

\section{Summary and application to the minimal ramification problem over $\mathbb{Q}$}
\label{sec:Q}
\label{sec:summary}

\subsection{Summary}
\label{subsec:summary}
We first summarize some of the 
results for the groups $G=S_n$ and 
$G=A_n$. These results are obtained by 
combining Theorem \ref{thm:main1} 
(proved in Section \ref{sec:thmprod} 
based on Abhyankar-type 
constructions) to handle the $n\ge 
p,\,n\neq p+1$ case; Theorem 
\ref{thm:tame2} (proved in 
Section \ref{sec:two} based on the tamely ramified rational function construction) to 
show $r_{\F_q(T)}(G)\le 2$ whenever 
$p>n$ or $n=p+1$; Theorem \ref{thm:tame} 
(combined with Proposition 
\ref{lem:Hqdm} which is used to 
ensure the requisite $H(q,d,e)$ 
conditions hold; both proved in Section \ref{sec:oneramified}) to eliminate the ramification at infinity and obtain $r_{\F_q(T)}(G)=1$ in the $p>n$ and $n=p+1$ cases, when applicable; a few remaining 
cases are handled ad hoc (some of 
the realizations were found by a computer search).

\begin{theorem}[Main results for $S_n$]\label{thm:main_for_sn}
Let $n\geq 2$ and $q=p^\nu$ a prime power.
Then $\ram_{\mathbb{F}_q(T)}(S_n)\leq 2$,
and $\ram_{\mathbb{F}_q(T)}(S_n)=1$ in each of the following cases:
\begin{enumerate}
\item $p<n-1$ or $p=n$  or $p=2$
\item $q\equiv 1\mbox{ mod } 4$ 
\item $q>(2n-3)^2$ 
\item The function field analogue of Schinzel's hypothesis H (Conjecture \ref{conj:SchinzelFF}) holds for $\F_q(T)$.
\end{enumerate}
\end{theorem}

\begin{proof}
The case $n=2$ is very easy and follows for example from Theorem~\ref{thm:abelian}, so assume from now on that $n\geq 3$.
If $p<n-1$, $p=n$ or $(n,p)=(3,2)$, then $\ram_{\mathbb{F}_q(T)}(S_n)=1$ by Theorem~\ref{thm:main1},
so suppose that $3\leq p=n-1$ or $p>n$.

For the rest of this proof, we call an $S_n$-extension of $\mathbb{F}_q(U)$
a $(q,n,m)$-realization
if it is the splitting field of $f-Uc$ with $f,c\in\F_q[X]$, ${\rm deg}(f)=n$, ${\rm deg}(c)=m<n$,
and $\Disc_X(f-Uc)$ a prime power.

We first prove $\ram_{\mathbb{F}_q(T)}(S_n)\leq 2$ in the remaining cases,
for which it suffices to exhibit a $(q,n,m)$-realization for some $m$:
For $n=5$ or $n\geq 7$,
Theorem~\ref{thm:tame2} with $m=n-4$ gives a $(q,n,n-4)$-realization
(note that $q\geq p\geq n-1>(2n-7)^{2/3}$ for all $n\geq4$).
For $n=3,4,6$, Proposition~\ref{prop:Morse2} gives a $(q,n,0)$-realization.

We now prove $\ram_{\mathbb{F}_q(T)}(S_n)=1$ in cases (2)-(4),
for which it suffices by Lemmas~\ref{lem:disjoint} and \ref{lem:eliminateinfty} to exhibit a $(q,n,m)$-realization such that $(p,n-m)=1$ and
$\HH(q,n+m-1,n-m)$ holds:

If (4) holds, then $\HH(q,d,e)$ always holds by Proposition~\ref{lem:Hqdm}(1),
so we can take the realizations from above.

If (3) holds
and $n\neq 4$,
Theorem~\ref{thm:tame2} with $m=n-2$ gives a $(q,n,n-2)$-realization (as $q>(2n-3)^2=(2m+1)^2$),
and $\HH(q,2n-3,2)$ always holds by Proposition~\ref{lem:Hqdm}(3) (as $q$ is odd).
If (3) holds and $n=4$, 
we take the $(q,4,0)$-realization from above,
and Proposition~\ref{lem:Hqdm}(5) gives $\HH(q,3,4)$ for every $q\geq (2^1\cdot 3-1)^2=25=(2n-3)^2$. 

If (2) holds,
then $\HH(q,d,4)$ holds for every odd $d$ by
Proposition~\ref{lem:Hqdm}(3);
if $n=5$ or $n\geq 7$,
we obtained a $(q,n,n-4)$-realization above,
and $\HH(q,2n-5,4)$ holds;
if $n=4$ we take the $(q,4,0)$-realization from above,
and $\HH(q,3,4)$ holds;
if $n=3$ or $n=6$ we are in case (3) as soon as $q> 9$ resp.~$q>81$; 
for $n=6$ and $q\equiv 1\mbox{ mod }3$,
$\HH(q,5,6)$ holds by Proposition~\ref{lem:Hqdm}(3)
so we can take the $(q,6,0)$-realization from above.
In each of the remaining cases, the following table provides 
$f,c\in\F_p[X]$ with $\deg f=n$, $\deg c=m<n$ and $h\in\F_p[T]$ with $\deg h=n-m$
such that 
$f'c-cf'$ is irreducible
and
$\mathcal{F}:=\Disc_X(f(X)-Tc(X))$ is such that $\mathcal{F}(h)$ is irreducible,
so that by the arguments in the proof of Theorem~\ref{thm:tame}(i) 
(where now $h$ is given explicitly instead of applying $\HH(q,n+m-1,n-m)$,
and $f$ is given explicitly instead of applying Proposition~\ref{prop:g}(i)),
the splitting field of $f(X)-h(T)c(X)$ over $\F_p(T)$
is geometric with Galois group $S_n$ and is ramified only over the prime $\mathcal{F}(h)$.

\newcolumntype{L}{>{$}l<{$}}
\begin{center}
\begin{tabular}{|ll|l|l|l|}
\hline
 $n$  & $p$  & $f$  & $c$  & $h$   \\
\hline
$3$ &$5$ & $X^3 + 1$ &$X+2$ &$T^2$ \\
$6 $&$5 $&$X^6 + 1 $&$X^2+X $&$T^4+1$  \\
$6 $&$17 $& $X^6 + X^2 + X$ & $1 $ & $T^6+T+2$ \\
$6 $&$29 $&$X^6 + X^2 + X $&$1 $&$T^6+T+6$  \\
$6 $&$41 $&$X^6+X $&$1 $&$T^6+T+1$  \\
$6 $&$53 $&$X^6+X^2+13X $&$1 $&$T^6+T$  \\
\hline
\end{tabular}
\end{center}
\end{proof}

\begin{remark}
Note that instead of $q\equiv 1\mbox{ mod }4$ we could also treat other arithmetic progressions:
For any prime number $\ell$, if $q\equiv 1\mbox{ mod }2\ell$, $n\geq 2\ell+3$, 
and $2n\not\equiv 1\mbox{ mod }\ell$,
let $m=n-2\ell$. 
Then, in the language of the previous proof,
Theorem~\ref{thm:tame2}
gives a $(q,n,m)$-realization,
and $\HH(q,n+m-1,n-m)=\HH(q,2n-2\ell-1,2\ell)$ holds by Proposition~\ref{lem:Hqdm}(3).
\end{remark}

\begin{theorem}[Main results for $A_n$]\label{thm:main_for_an}
Let $n\geq 3$ and $q=p^\nu$ a prime power.
If $p>2$ 
or 
$\F_q\supseteq\F_{4}$, then $\ram_{\mathbb{F}_q(T)}(A_n)\leq 2$,
and $\ram_{\mathbb{F}_q(T)}(A_n)=1$ in each of the following cases:
\begin{enumerate}
\item $2<p<n-1$ or $p=n$ or $p=n-1,\F_q\supseteq\F_{p^2}$ or $p=3,n=4$.
\item $p=2$ and $\F_q\supseteq\F_4$ or $10\neq n\ge 8,n\equiv 0,1,2,6,7\pmod 8$ 
\item $12\neq n\ge 10$ and the function field analogue of Schinzel's hypothesis H (Conjecture \ref{conj:SchinzelFF}) holds for $\F_q(T)$.
\end{enumerate}
\end{theorem}

\begin{proof} 
The cases $p=2$ and $p=3$ follow from Theorem~\ref{thm:main1},
so 
henceforth we will assume $p>3$.

If $p\le n,p\neq n-1$ or $p=n-1,\F_q\supseteq\F_{p^2}$, all assertions follow from Theorem~\ref{thm:main1}. 

If $p=n-1$ and $n\ge 14$ we may use Theorem~\ref{thm:tame2} with $m=3$ to obtain $\ram_{\F_q(T)}(A_n)\le 2$ in this case. 
Assuming Conjecture $\ref{conj:SchinzelFF}$ allows us to apply Theorem~\ref{thm:tame} combined with Proposition~\ref{lem:Hqdm}(1) (once again with $m=3$) to conclude $\ram_{\F_q(T)}(A_n)=1$ in this case.

If $p=n-1$ and $n\in\{6,8,12\}$,
the following table provides irreducible $f\in\F_p[T,X]$ 
with $\deg f=n$ and $\Disc_X(f)$  a square in $\F_p(T)$ with only one prime divisor $T$.
Computer verification shows that ${\rm Gal}(f/\F_p(T))$
contains a $3$-cycle and an $(n-1)$-cycle, hence by Theorem~\ref{thm:jones} contains $A_n$.
Thus the splitting field of $f$ over $\F_q(T)$ is ramified at most over $T$ and the infinite prime,
and ${\rm Gal}(f/\F_q(T))=A_n$ as $A_n$ is simple.
\begin{center}
\begin{tabular}{|ll|l|l|}
\hline
$n$ &$p$ &$f$ & $\Disc_X(f)$  \\
\hline
$6$ &$5$ & $X^6+X^5T - 2X^3T^3 + XT + T^2$ & $4T^{18}$ \\
$8$ & $7$ & $X^8 + 3X^2 + XT - 2$ & $4T^2$  \\
$12$ & $11$ & $X^{12} + 5XT^3 - 5X^2 - 2$ & $4T^6$ \\
\hline
\end{tabular}
\end{center}

Now assume that $p>n$.
If $n\ge 13$ or $n=11$ we apply Theorem~\ref{thm:tame2} 
with $m=2$ or $m=3$ chosen so that $m\not\equiv n\pmod 2$
and thus obtain $\ram_{\F_q(T)}(A_n)\le 2$ in this case. 
If $n=10$ we apply Theorem~\ref{thm:tame2} with $m=1$ (note that $\left(\frac{q}{n-m}\right)=\left(\frac{q}{3}\right)^2=1$).
Assuming Conjecture $\ref{conj:SchinzelFF}$ allows us to apply Theorem~\ref{thm:tame} combined with Proposition~\ref{lem:Hqdm}(1) (with $m$ as above) to conclude $\ram_{\F_q(T)}(A_n)=1$ in these cases.
In all other cases, Theorem~\ref{thm:main2} gives that 
$\ram_{\F_q(T)}(A_n)\le 2$.
\end{proof}

\begin{remark}\label{remark:cfsg_details} The proofs of Theorems \ref{thm:main_for_sn} and \ref{thm:main_for_an} above rely on Theorems \ref{thm:tame2} and \ref{thmlist}, which use the Classification of Finite Simple Groups (CFSG) for some of the cases. By Remark \ref{remark:manning}, the CFSG is not necessary in the proof of Theorem \ref{thm:tame} provided $n-m\le 15$. The proof of Theorem \ref{thmlist} uses the CFSG only in the cases $n=p+2$ and $G=A_n,p=2|n$. Hence Theorems \ref{thm:main_for_sn}, \ref{thm:main_for_an}, and by implication Theorems \ref{thm:Sn}, \ref{thm:An}, \ref{thm:main1}, \ref{thm:abhyankar}, require the CFSG only in the following cases: \begin{itemize}\item$n=p+2$\item $G=A_n$ and ($p>n$ or ($n=p+1$ and $\F_q\not\supset\F_{p^2})$ or $p=2|n$).\end{itemize} The upper bound $r_{\F_q(T)}(A_n)\le 2$ in the case $p>n$ follows from Theorem \ref{thm:main2}, which does not require the CFSG.  
\end{remark}

\subsection{Application to the minimal ramification problem over $\mathbb{Q}$}
Finally, we apply our results for $S_n$ over $\mathbb{F}_q(T)$
to give a conditional proof of Conjecture~\ref{conj:BM} for $S_n$ over $\mathbb{Q}$:

\begin{lemma}\label{lem:lift_real}
Let $n\in\mathbb{N}$, $h_0\in\mathbb{Z}$, and $S$ a finite set of primes numbers,
and for each $p\in S$ let $f_p\in\mathbb{Z}[X]$ be monic of degree $n$.
There exist $f\in\mathbb{Z}[X]$ monic of degree $n$ and $c\in\mathbb{Z}$ such that $f\equiv f_p\pmod p$ and $c\equiv 1\pmod p$
for every $p\in S$,
and $f-c^nh_0$ has $n$ roots in $\mathbb{R}$.
\end{lemma}

\begin{proof}
By the Chinese Remainder Theorem there exists a monic $f_0\in\mathbb{Z}[X]$
with $f_0\equiv f_p\pmod p$ for every finite $p\in S$.
Let $\pi=\prod_{p\in S}p$,
write $f_0=\sum_{i=0}^na_iX^i$ and 
let $f_\infty\in\mathbb{Z}[X]$ be any monic polynomial of degree $n$
with $n$ roots in $\mathbb{R}$.
Since fractions of the form $\frac{x\pi}{y\pi+1}$ with $x,y\in\mathbb{Z}$ are dense in $\mathbb{R}$, 
we can choose
$x_i,y_i\in\mathbb{Z}$ such that
$\tilde{f}:=X^n+\sum_{i=0}^{n-1}(a_i+\frac{x_i\pi}{y_i\pi+1})X^i$
is arbitrarily close to $f_\infty+h_0$, 
in particular so close that also $\tilde{f}-h_0$ has $n$ roots in $\mathbb{R}$.
Then with $c=\prod_{i=0}^{n-1}(y_i\pi+1)$,
the polynomial
$f(X):=c^n\tilde{f}(c^{-1}X)\in\mathbb{Z}[X]$ satisfies the claim.
\end{proof}

\begin{theorem}\label{thm:S_n_over_Q}
Schinzel's hypothesis H (Conjecture \ref{conj:Schinzel}) implies that
$\ram_\mathbb{Q}(S_n)=1$ for every $n\geq2$, i.e.~Conjecture \ref{conj:BM} holds for symmetric groups.
\end{theorem}

\begin{proof}
In case $n=2$, for example $L=\mathbb{Q}(\sqrt{2})$ does the job,
so assume that $n\geq 3$.
Let
$S_0$ denote the set of prime numbers $\ell<d:=n(n-1)$
and for $\ell\in S_0$ fix a separable polynomial $f_\ell\in\mathbb{F}_\ell[X]$.
By Proposition~\ref{thm:large_q},
if $p$ is sufficiently large with respect to $n$,
there exists a monic $f_p\in\mathbb{F}_p[X]$ of degree $n$ such that
the splitting field of $f_p-U$ over $\mathbb{F}_p(U)$
is geometric with Galois group $S_n$, and
$\mathcal{F}=\disc_X(f_p-U)\in\mathbb{F}_p[U]$ is irreducible of degree $n-1$
(cf.~(\ref{eq:disc2}) and Lemma~\ref{lem:crit}).
By Proposition~\ref{lem:Hqdm}(2,c),
if $p$ is sufficiently large with respect to $n$, 
there exists a monic $h_p\in\mathbb{F}_p[T]$
with ${\rm deg}(h_p)=n$
such that $\mathcal{F}(h_p(T))$ is irreducible.
Fix such a sufficiently large $p\notin S_0$ and $f_p,h_p$ as above,
and let $g_p:=f_p(X)-h_p(T)$.
Then $\disc_X(g_p)=\mathcal{F}(h_p(T))\in\mathbb{F}_p[T]$ is irreducible,
and Lemma~\ref{lem:disjoint} implies that
the splitting field of $g_p$ 
over $\mathbb{F}_p(T)$ is geometric with Galois group $S_n$.

Choose $h\in\mathbb{Z}[T]$ monic of degree $n$ with $h\equiv h_p\pmod p$
and $h\equiv T^n\pmod\ell$ for $\ell\in S_0$.
By Lemma~\ref{lem:lift_real}
there exist a monic $f\in\mathbb{Z}[X]$ and $c\in\mathbb{N}$ 
such that $f\equiv f_\ell\pmod\ell$ and $c\equiv 1\mbox{ mod }\ell$ for every $\ell\in S_0\cup\{p\}$,
and such that $f-c^nh(0)$ has $n$ roots in $\mathbb{R}$.
Let 
$$
 g(X,T)=f(X)-c^nh(c^{-1}T)
$$ 
and note that
$g\in\mathbb{Z}[X,T]$ is monic of degree $n$ both in $X$ and in $T$,
$g\equiv g_p\pmod p$ and $g\equiv f_\ell-T^n\pmod\ell$ for $\ell\in S_0$.
So as ${\rm Gal}(g_p/\overline{\mathbb{F}}_p(T))=S_n$ 
and ${\rm deg}_X(g)={\rm deg}_X(g_p)=n$, we get that 
the splitting field  $L$ of $g$ over $\mathbb{Q}(T)$
is geometric with Galois group $S_n$:
Reduction modulo $p$ extends to a place from $\overline{\mathbb{Q}}$ to $\Fb_p$ by Chevalley's theorem,
and further to a place
from $\overline{\mathbb{Q}}(T)$ to $\Fb_p(T)$ mapping $T$ to $T$,
so ${\rm Gal}(g_p/\overline{\mathbb{F}}_p(T))=S_n$
embeds into ${\rm Gal}(g/\overline{\mathbb{Q}}(T))$
(cf.~\cite[Lemma 16.1.1(a)]{FJ}),
which is itself a subgroup of $S_n$.
Similarly, also $\disc_X(g)\equiv\disc_X(g_p)\mbox{ mod }p$,
so as $\disc_X(g_p)$ is irreducible and 
${\rm deg}_T(\disc_X(g))=d={\rm deg}_T(\disc_X(g_p))$,
we get that $\disc_X(g)$ is irreducible.
In particular, $\disc_X(f-U)$ is irreducible.
Identifying $U=c^nh(c^{-1}T)\in\mathbb{Q}(T)$, 
since $\disc_X(f-U)$ is irreducible, 
the splitting field of $f-U$ over $\mathbb{Q}(U)$ is ramified at most over the finite prime $\mathcal{F}_0:=\disc_X(f-U)$ and the infinite prime,
so by Lemma~\ref{lem:eliminateinfty}, 
$L/\mathbb{Q}(T)$ is ramified only at the prime $\mathcal{F}:=\mathcal{F}_0(c^nh(c^{-1}T))=\Delta_X(g)$.

In geometric terms, the extension $L/\mathbb{Q}(T)$ corresponds
to a branched covering of geometrically irreducible smooth curves $\phi\colon C\rightarrow\mathbb{P}^1_\mathbb{Q}$
with ${\rm Gal}(\mathbb{Q}(C)/\mathbb{Q}(\mathbb{P}^1))=S_n$.
We wish to apply \cite[Proposition 1.11]{BSS} to $\phi$,
for which we use the notation and definitions from there.
As $\phi$ has only the one ramified prime divisor $\mathcal{F}$, its branch locus is the zero locus of an irreducible homogeneous polynomial $D(X,Y)\in\mathbb{Z}[X,Y]$ of degree $d$.
As in \cite[p.~923]{BSS}, 
for $\xi\in\mathbb{P}^1(\Q)$
we denote by $A^\phi_{\xi}$ the specialized algebra at $\xi$,
which for $\xi=[a:1]\in\mathbb{A}^1(\Q)$ is
$A^\phi_{[a:1]}=R\otimes_{\Q[T]_{(T-a)}}\Q=R/(T-a)$,
where $R$ denotes the integral closure of $\Q[T]_{(T-a)}$ in $\Q(C)$.
As $\xi$ is unramified in $\Q(C)$, $A^\phi_{\xi}$ is an \'etale $\Q$-algebra, i.e.~$A^\phi_{\xi}\cong\prod_{i=1}^rF_i$
with $F_i/\Q$ a field extension,
and a prime of $\Q$ is called unramified in $A^\phi_{\xi}$ if it is unramified in each $F_i$.
We now show that the set of universally ramified primes $U(\phi)$ as defined in \cite[p.~924]{BSS} is empty,
i.e.~every prime of $\Q$ 
is unramified in some $A^\phi_{\xi}$:

As $g(X,0)=f-c^nh(0)$ is separable,
the integral closure $R$ of $\Q[T]_{(T)}$ in $\Q(C)$ is generated by the roots of $g$ (cf.~\cite[Lemma 6.1.2]{FJ}),
hence $A^\phi_{[0:1]}$
is generated by elements $a_1,\dots,a_n$ with $g(a_i,0)=0$.
In particular, since $g(X,0)$ splits into linear factors over $\mathbb{R}$, 
the infinite prime is unramified in $A^\phi_{[0:1]}$.
For prime numbers $\ell<d$, $g(X,0)\equiv f_\ell\mbox{ mod }\ell$ is separable,
so $\ell$ is unramified in each $\Q(a_i)$ (see e.g.~\cite[Prop.~I.8.3]{Neukirch}) and hence in $A^\phi_{[0:1]}$.
Finally, as $\disc_X(g)$ is of degree $d$ with leading coefficient $\pm n^n$, it is also not the zero function
modulo any prime number $\ell> d$,
so for every such $\ell$ there exists $a\in\mathbb{Z}$ such that
$g(X,a)\mbox{ mod }\ell$ is separable, 
and thus the same argument shows that $\ell$ is unramified in
$A^\phi_{[a:1]}$.

Thus in the language of \cite[Definition 1.10]{BSS},
$G=S_n$ has a $(U;\mathbf{d})$-realization with $U=\emptyset$ and $\mathbf{d}=(d)$.
Therefore, \cite[Proposition 1.11]{BSS} gives a realization of $G$
over $\mathbb{Q}$
with at most $B(\mathbf{d})+\#U$ many ramified primes,
and under Schinzel's hypothesis H, $B(\mathbf{d})\leq 1$, cf.~\cite[(13) on p.~923]{BSS}.
\end{proof}

\begin{remark}\label{remark:plans}
Plans \cite[Remark 3.10]{Plans} sketches how to use Schinzel's hypothesis H (Conjecture \ref{conj:Schinzel}) to construct $S_n$-extensions $L$ of $\mathbb{Q}$ ramified in only one finite prime. However, he constructs $L$ as the splitting field of an irreducible trinomial $f=X^n+aX^i+b$, and therefore $L$ is not totally real if $n\geq5$: 
Indeed, replacing $X$ by $X^{-1}$ if necessary we may assume that $i\geq\frac{n}{2}$, 
so since the derivative $f'=X^{i-1}(nX^{n-i}+ia)$ has at most $n-i+1$ distinct roots, in particular at most $n-i+1$ distinct roots in $\mathbb{R}$,
Rolle's theorem shows that $f$ has at most $n-i+2\leq\frac{n}{2}+2<n$ roots in $\mathbb{R}$.
\end{remark}

\section*{Acknowledgements}
\noindent
The authors would like to thank 
Stephen Cohen for helpful discussion and for pointing out the results in \cite{Cohen},
David Harbater and Robert Lemke-Oliver for some helpful discussions,
Joachim K\"onig for pointing out the observation in Remark \ref{rem:joachim} and providing a reference for Lemma~\ref{lem:gen_by_cyc},
and Danny Neftin for helpful discussions on \cite{KisilevskyNeftinSonn}.

The computations in Section \ref{sec:twinprimes} were carried out with SageMath.
The computations in Section \ref{subsec:summary} were carried out with SageMath and verified using Magma.

L.~B.-S.\ was partially supported by a grant of the Israel Science Foundation no.~702/19. 
A.~E. was partially supported by a grant of the Israel Science Foundation no.~2507/19.
Some of this research was conducted while A.~F.\ was visiting Tel Aviv University, and he would like to thank 
Ze\'ev Rudnick and the Shaoul Fund for financial support.

\end{document}